\documentclass{amsart}

\usepackage{amssymb,amsmath,amsxtra,amsthm,mathtools,ifthen}
\usepackage[all]{xy}
\usepackage{graphicx}
\usepackage{enumerate}
\usepackage{hyperref}	

\makeatletter
\renewenvironment{proof}[1][\proofname]{\par
  \pushQED{\qed}%
  \normalfont \topsep6\p@\@plus6\p@\relax
  \trivlist
  \item[\hskip\labelsep
        \textsc{#1.}]\ignorespaces
}{%
  \popQED\endtrivlist\@endpefalse
}
\makeatother

\newcommand{\addtheorem}[3]{
	\newtheorem{#1}{#3}							
	\makeatletter								
	\expandafter\def\csname c@#1\endcsname{\csname c@#2\endcsname}		
	\expandafter\def\csname p@#1\endcsname{\csname p@#2\endcsname}
	\makeatother								
	\expandafter\def\csname the#1\endcsname{\csname the#2\endcsname}	
	\expandafter\def\csname #1autorefname\endcsname{#3}			
	}

\newtheoremstyle{joltthm}	
		{\baselineskip}	
		{\baselineskip}	
		{\itshape}	
		{}		
		{\bfseries}	
		{.}		
		{ }		
		{}		

\newtheoremstyle{joltdefn}	
		{\baselineskip}	
		{\baselineskip}	
		{}		
		{}		
		{\bfseries}	
		{.}		
		{ }		
		{}		

\theoremstyle{joltthm}
\newtheorem{theorem}{Theorem}[section]
\addtheorem{lemma}{theorem}{Lemma}
\addtheorem{proposition}{theorem}{Proposition}
\addtheorem{corollary}{theorem}{Corollary}

\theoremstyle{joltdefn}
\addtheorem{definition}{theorem}{Definition}
\addtheorem{remark}{theorem}{Remark}
\addtheorem{remarks}{theorem}{Remarks}

\addtheorem{example}{theorem}{Example}
\addtheorem{examples}{theorem}{Examples}

\numberwithin{equation}{section}

\newcommand{\ie}{\emph{i.e.}\ }
\newcommand{\via}{\emph{via}\ }

\newcommand{\etc}{\emph{etc.}\ }
\newcommand{\cf}{\emph{cf.}\ }

\newcommand{\odd}{\text{odd}}

\newcommand{\NN}{\mathbb N}
\newcommand{\RR}{\mathbb R}
\newcommand{\ZZ}{\mathbb Z}

\newcommand{\Lie}{\mathcal L}
\newcommand{\XX}{\mathfrak X}

\newcommand{\Diff}{\mathrm{Diff}}
\newcommand{\supp}{\mathop{\rm supp}\nolimits}
\newcommand{\OO}{\mathcal{O}}

\newcommand{\Ber}{\mathop{\rm Ber}\nolimits}

\newcommand{\id}{\mathop{\rm id}\nolimits}
\newcommand{\pr}{\mathop{\rm pr}\nolimits}
\newcommand{\sgn}{\mathop{\rm sgn}\nolimits}
\newcommand{\codim}{\mathop{\rm codim}\nolimits}

\newcommand{\defi}{\coloneqq}			

\newcommand{\Wedge}{{\textstyle\bigwedge}}

\ifdefined\fr
	\errmessage{\string\fr\space is already defined}
\else
	\def\fr #1/#2{\ensuremath{\frac{#1}{#2}}}
\fi

\ifdefined\lr(
	\errmessage{\string\lr(\space is already defined}
\else
	\def\lr(#1){{\left(#1\right)}}
\fi

\ifdefined\bg(
	\errmessage{\string\bg(\space is already defined}
\else
	\def\bg(#1(#2)#3){{\big(#1(#2)#3\big)}}
\fi

\ifdefined\diff
	\errmessage{\string\diff\space is already defined}
\else
	\def\diff#1^#2{\ensuremath{\partial_{#1}^{#2}}}
\fi

\ifdefined\der
	\errmessage{\string\der\space is already defined}
\else
	\def\der#1/#2{\ifthenelse	{\equal{#1}{}}
								{\ensuremath{\partial_{#2}{#1}}}
								{\ensuremath{\frac{\partial #1}{\partial #2}}}
					}
\fi

\ifdefined\derf
	\errmessage{\string\derf\space is already defined}
\else
	\def\derf#1/#2{\ifthenelse	{\equal{#1}{}}
								{\ensuremath{\frac{\partial #1}{\partial #2}}}
								{\ensuremath{\partial_{#2}{#1}}}
					}
\fi

\ifdefined\D
	\errmessage{\string\D\space is already defined}
\else
	\def\D#1/#2{\ensuremath{\frac{D#1}{D#2}}}
\fi

\ifdefined\Da
	\errmessage{\string\Da\space is already defined}
\else
	\def\Da#1/#2{\ensuremath{\frac{|D#1|}{|D#2|}}}
\fi

\ifdefined\multi
	\errmessage{\string\multi\space is already defined}
\else
	\def\multi(#1,#2){\ifthenelse	{\equal{#1}{0}}
									{{\mathbb Z}_2^{#2}}
									{\ifthenelse{\equal{#2}{0}}
												{{\mathbb N}_0^{#1}}
												{\ensuremath{{{\mathbb N}_0^{#1}\!\times\!{\mathbb Z}_2^{#2}}}}
									}
					}
\fi

\newcommand{\Cref}[1]{Equation~\eqref{#1}}

\allowdisplaybreaks

\begin{document}

\begin{abstract}
	We investigate the Berezin integral of non-compactly supported quantities. In the framework of supermanifolds with corners, we give a general, explicit and coordinate-free repesentation of the boundary terms introduced by an arbitrary change of variables. As a corollary, a general Stokes's theorem is derived---here, the boundary integral contains transversal derivatives of arbitrarily high order.
	
	
\end{abstract}

	\title{Integration on Non-Compact Supermanifolds}
	\author{Alexander Alldridge}
	\address{Mathematisches Institut\\ÊUniversit\"at zu K\"oln\\ Weyertal 86--90\\ 50931 K\"oln\\ Germany}
	\email[A.~Alldridge]{alldridg@math.uni-koeln.de}
	\author{Joachim Hilgert}
	\address{Institut f\"ur Mathematik\\ÊUniversit\"at Paderborn\\ Warburger Str.~100\\ 33098 Paderborn\\ Germany}
	\email[J.~Hilgert]{hilgert@math.upb.de}	
	\author{Wolfgang Palzer}
	\address{Mathematisches Institut\\ÊUniversit\"at zu K\"oln\\ Weyertal 86--90\\ 50931 K\"oln\\ Germany}
	\email[W.~Palzer]{palzer@math.uni-koeln.de}

	\thanks{This research was supported by the Leibniz junior independent research group grant and SFB/Transregio 12 ``Symmetries and Universality in Mesoscopic Systems'', funded by Deutsche Forschungsgemeinschaft (DFG)}

\maketitle

\section{Introduction}

Supermanifolds were introduced by Berezin, Leites and Konstant in the 1970s as a mathematical framework for the quantum theory of commuting and anticommuting fields. A remarkable contribution was Berezin's definition of his integral, in Ref.~\cite{Be66}, predating the definition of supermanifolds by several years, and providing at the time sufficient indication that a reasonable supersymmetric analysis should exist. 

Despite its utility, the integral suffers from a fundamental pathology: Only the integral of \emph{compactly supported} quantities is well-defined in a coordinate independent form---changes of variables introduce, in general, so-called \emph{boundary terms}. This can be seen as a major obstacle in the development of global superanalysis. 

For example, although Stokes's theorem 
\begin{equation}\label{eq:stokes-intro}
	\int_Md\omega=\int_{\partial M}\omega
\end{equation}
has been extended to supermanifolds by Bernstein and Leites \cite{BL77}, this extension supposes that the supermanifold structure on the boundary $\partial M$ enjoys a rather strong compatibility requirement. In fact, even for compactly supported integrands $\omega$, the conclusion of the theorem \emph{fails} in general, unless this assumption is made (\cf \autoref{ex:4.6} below). 

An invariant definition of the integral can however be made, on the basis of the following simple observation: For any supermanifold $M$, there exist morphisms $\gamma\colon M\to M_0$---which we call \emph{retractions}---which are left inverse to the canonical embedding $j_M\colon M_0\to M$. Any retraction $\gamma$ is a submersion whose fibres have compact base; thus, there is a well-defined \emph{fibre integral} $\gamma_!$ which takes Berezin forms on $M$ to volume forms on $M_0$, and one may define 
\begin{equation}\label{eq:intretractdef}
\int_{(M,\gamma)}\omega=\int_{M_0}\gamma_!(\omega)\ . 
\end{equation}

Taking pullback retractions, this definition is now trivially well-defined under coordinate changes. Furthermore, whereas retractions are non-unique in general, for certain classes of supermanifolds---\emph{e.g.}, Lie supergroups $G$, homogeneous $G$-supermanifolds, and superdomains---there exist canonical retractions.

This framework allows us to give an explicit description of the behaviour of the integral under coordinate changes. To state our main result (\autoref{thm:trafo}), let $N\subset M^{p|q}$ be an open subspace of a supermanifold whose underlying space $N_0\subset M_0$ is a manifold with corners. That is, we have $N_0=\{\rho_i>0\mid i=1,\dotsc,n\}$ for some functions $\rho_i$ which define \emph{boundary manifolds} $H_0=\{\rho_{i_1}=\dotsm=\rho_{i_k}=0,\rho_j>0\ (j\neq i_m)\}$. Let $\gamma$, $\gamma'$ be retractions on $N$. On each $H_0$, one considers the supermanifold structure $H$ induced by $\gamma^*(\rho_{i_m})$ and the retraction $\gamma_H$ induced by $\gamma$. Let $D_i$ be even vector fields such that $D_i(\gamma^*(\rho_j))=\delta_{ij}$ on suitable neighbourhoods of $\{\gamma^*(\rho_i)=\gamma^*(\rho_j)=0\}$. 

Then, for any Berezin density $\omega$ such that the integrals exist,
\begin{align*}
	\int_\lr(N,\gamma')\omega=\int_\lr(N,\gamma)\omega+\sum_{H\in B(\gamma^*(\rho))}\sum_{j\in J_H}\pm\int_\lr(H,\gamma_H)\big(\omega_j.D^{j\downarrow}\big)\big|_{H,\gamma^*(\rho)}.
\end{align*}

Here, we sum over all $H=\{\gamma^*(\rho_{i_1})=\dotsm=\gamma^*(\rho_{i_k})=0\}$ and all multi-indices $j\in J_H=\NN^{\{i_1,\dotsc,i_k\}}$; moreover, $\omega_j\defi\fr 1/{j!}\left(\gamma'^*(\rho)-\gamma^*(\rho)\right)^j\omega$ and $j\mathord\downarrow$ denotes the multi-index $j$ with entries reduced by one. The differential operators on the right hand side are of degree up to $\frac q2$. 

From this change of variables formula, we deduce a version of Stokes's theorem which is valid for an arbitrary supermanifold structure on the boundary (\autoref{cor:stokes2}). Compared to \Cref{eq:stokes-intro}, the right hand side depends not only on $\omega|_{\partial M}$, but on transversal derivatives up to order $\frac q2$.

The question of defining the integral of non-compactly supported Berezinians was first studied by Rothstein \cite{Ro87} in his seminal paper. His fundamental insight was that the integral becomes well-defined if instead of the Berezinian sheaf, one considers the sheaf of super-differential operators with values in volume forms. This insight is vital---indeed, Rothstein's techniques form the basis of our investigations, and one may view \Cref{eq:intretractdef} as an attempt to translate Rothstein's definition of the Berezin integral \emph{via} the `Fermi integral' to the realm of ordinary Berezinians. 

For applications to superanalysis, Rothstein's sheaf is somewhat unwieldy, since it is an $\mathcal O_M$-module of infinite rank. For example, in the context of homogeneous supermanifolds, one frequently fixes integrands by invariance. Of course, this can only be done for $\mathcal O_M$-modules of rank one, which favours the Berezinian sheaf as a tool for superanalysis. 

The applications we have in mind come from the spherical harmonic analysis on Riemannian symmetric supermanifolds, in particular, the study of orbital and Eisenstein integrals in the spirit of Harish-Chandra. Besides its relation to representation theory \cite{A11}, this subject is of high current interest in mathematical physics, in the study of $\sigma$-model approximations of invariant random matrix ensembles, as are applied to disordered metals and topological insulators \cite{Zir91,HHZ05,LPS10,DSZ10,GLMZ11}. 

\smallskip
Let us end with a brief synopsis of our paper. In \autoref{integraldef}, we recall some basic facts and define the integral of Berezin densities with respect to a retraction. In \autoref{stokes}, we prove a version of Stokes's theorem in this setting (\autoref{thm:stokes}). Here, the supermanifold structure on the boundary has to be chosen compatibly (see below). In \autoref{coordchange}, we prove a version of our change of variables formula in terms of coordinates (\autoref{thm:5.4}). Here, the `boundary' nature of the `boundary terms' is not yet evident. This is finally accomplished in \autoref{boundary_smfd}, where the language and technique of supermanifolds with corners and boundary supermanifolds is introduced; here, the point of view of retractions proves particularly fruitful. By applying this machinery, we prove our main result (\autoref{thm:trafo}) and illustrate its use in some examples. Finally, we deduce a generalised Stokes's theorem (\autoref{cor:stokes2}) where the supermanifold structure on the boundary is arbitrary.

\medskip
\emph{Acknowledgements}. The present paper is based on the diploma thesis \cite{Pal10} of W.P. under the guidance of J.H. 

\section{The Berezin integral in the non-compact case}\label{integraldef}

We use the standard definition of supermanifolds in terms of ringed spaces. For basic facts on these, we refer the reader to \cite{Le80,CdG94}. Let us fix our notation. Given an object in the graded category, we will denote the underlying ungraded object by a subscript $0$. We denote supermanifolds as $M=(M_0,\OO_M)$, $N=(N_0,\OO_N)$, \etc Unless the contrary is stated explicitly, we will assume $M$, $N$ to be of dimension $(p,q)$. Manifolds will always be Hausdorff and second countable. By writing $U\subseteq M$ we will mean that $U$ is the ringed subspace $M|_{U_0}\defi(U_0,\OO_M|_{U_0})$ of $M$ given  by the open subset $U_0\subseteq M_0$. Thus, unions and finite intersections of open subspaces are defined. Further, the set of superfunctions $\OO_M(U_0)$ on $U$ is abbreviated by $\OO(U)$. Morphisms of supermanifolds $M\to N$ are denoted $\varphi=(\varphi_0,\varphi^*)$, with underlying smooth map $\varphi_0\colon M_0\to N_0$ and the sheaf morphism  $\varphi^*\colon\OO_N\to{\varphi_0}_*\OO_M$. For a given supermanifold $M$ we denote the canonical embedding by $j_M\colon M_0\to M$. Given $f\in\OO(U)$, we write $f_0$ for $j_M^*(f)$. 

Now we introduce a certain type of morphisms which will be central for the following developments.

\begin{definition}
	A morphism $\gamma\colon M\to M_0$ is called a \emph{retraction} if it is a right inverse of the canonical embedding $j_M$, \ie
	\begin{align*}
		\gamma\circ j_M=\id_{M_0}.
	\end{align*}
\end{definition}

\begin{remarks}
		In the literature, the subalgebra $\text{Im}\gamma^*\subseteq\OO(M)$ is called a \emph{function factor}. 

		It is a known fact that retractions always exist on (real) supermanifolds \cite[Lemma~3.2]{RS83}. However, they are in general not unique. For superdomains there exists a canonical choice of retraction. Using exponential charts, one may also give canonical retractions in the case of Lie supergroups; this can also be extended to the case of homogeneous supermanifolds. 
\end{remarks}

We will repeatedly use the following standard fact \cite[Theorem~2.1.7]{Le80}.

\begin{proposition}\label{prop:3.3}
	Let $M$, $N$ be supermanifolds, $y=(v_1,\dotsc,v_p,\eta_1,\dotsc,\eta_q)$ a coordinate system on $N$, and $x=(u_1,\dotsc,u_p,\xi_1,\dotsc,\xi_q)$ a family of superfunctions on $M$ where the $u_i$ are even and the $\xi_j$ are odd. Then there exists a unique morphism $\varphi\colon M\to N$ such that $\varphi^*(y)=x$, if and only if the function $(u_{1,0},\dotsc,u_{p,0})$ takes its values in $v_0(N_0)=\big\{(v_{1,0}(o),\dotsc,v_{p,0}(o))\,\big|\, o\in N_0\big\}$. 
\end{proposition}

\begin{definition}
	If $\gamma$ is a retraction and $u_0=(u_{1,0},\dotsc,u_{p,0})$ is a classical coordinate system, then $\gamma^*(u_0)$ is the even part of a coordinate system. 
	
	Conversely, if $s=(u,\xi)$ is a coordinate system, then there is a unique retraction $\gamma$ such that $\gamma^*(u_0)=u$, by the above proposition. We call this the \emph{retraction associated with $u$} (or $x$). 
	
	Let $x=(u,\xi)$ be a coordinate system and $\gamma$ be the retraction associated with $u$. Any superfunction $f$ possesses a unique decomposition 
\begin{align}\label{eq:1}
	f=\sum_{\hidewidth\nu\in\multi(0,q)\hidewidth}\gamma^*(f_\nu)\xi^\nu, \qquad f_\nu\in C^\infty,
\end{align}
where $\xi^\nu=\xi_1^{\nu_1}\dotsm\xi_q^{\nu_q}$.  Observe that in the literature, one commonly writes this expansion in terms of functions $g_\nu(u)$, where $g_\nu$ are functions on the range of the chart associated with $u_0$. We note further that $j_M^*(f)=f_\lr(0,\dotsc,0)$, which explains the abbreviation $f_0$ for $j_M^*(f)$.

Using decomposition \eqref{eq:1} we define derivations along the coordinates, 
\begin{alignat*}{4}
	\der f/{x_i}	&\defi\der f/{u_i}		&&\defi\ \sum_\nu\ \gamma^*\lr(\der f_\nu/{u_{i,0}})\xi^\nu,															&& i=1,\ldots,p,\\
	\der f/{x_{p+j}}&\defi\der f/{\xi_j}	&&\defi\ \sum_{\hidewidth\nu\,,\,\nu_j=1\hidewidth}\ \gamma^*(f_\nu)\xi^{\nu-e_j}(-1)^{\nu_1+\dotsc+\nu_{j-1}},\qquad	&& j=1,\ldots,q.
\end{alignat*}

We abbreviate $\der/{x_i}\defi\derf/{x_i}$ for $i=1,\ldots,p+q$ and $\diff x^i\defi\diff {x_{p+q}}^{ i_{p+q}}\circ\dotsb\circ\diff x_1^{i_1}$
for $i\in\multi(p,q)$. Corresponding abbreviations for $u$ and $\xi$ are similarly defined.

\end{definition}

\begin{definition}
	Let $M$ be a supermanifold and $(U,x)$ a local coordinate system. A \emph{Berezin form} $\omega$ on $U$ is an object of the form
	\begin{align*}
		\omega=fDx=(-1)^{|f|b(p,q)}Dx\, f,
	\end{align*}
	where $f$ is a superfunction on $U$. We make here no choice for the parity $|Dx|\defi b(p,q)$; a common one is $b(p,q)=p+q$. If $y=(v,\eta)$ is another coordinate system on $U$, then one requires 
	\begin{align*}
		\omega=fDx=f\D x/y Dy.
	\end{align*}
	
	Here, the Berezinian of the coordinate change is given by
	\begin{align*}
		\D x/y\defi	\Ber\lr(\der x/y)
				=\Ber\begin{pmatrix}
						\der u/v	&\der\xi/v\\
						\der u/\eta &\der\xi/\eta
					\end{pmatrix},
	\end{align*}
	where
	\begin{align*}
		\Ber\begin{pmatrix}
			R & S\\
			T & V
		\end{pmatrix}
		&\defi\det(R-SV^{-1}T)\det V^{-1}.
	\intertext{\indent The correspondence $U_0\mapsto\Ber U$ extends to an $\mathcal O_M$-module sheaf $\Ber_M$ on $M_0$, as is well-known \cite{Man97,Ch94,AH10}. The $\OO(M)$-module of Berezin forms on $M$ is denoted by $\Ber M$.
	\newline\indent One defines \emph{Berezin densities} similarly, replacing the character $\Ber$ by}
		|\Ber|\begin{pmatrix}R&S\\ T&V\end{pmatrix}
		&\defi\sgn j_M^*(\det R)\cdot\Ber\begin{pmatrix}R&S\\T&V\end{pmatrix}.
	\end{align*}
	
	Thus, Berezin densities have the local form $f|Dx|=(-1)^{|f||Dx|}|Dx|\, f,$ and obey the transformation law
	\begin{align*}
		f|Dx|=f\Da x/y|Dy|,
	\end{align*}
	where $\Da x/y=\pm \D x/y$ according to the relative orientation of $u_0$ and $v_0$. The corresponding $\OO_M$-module sheaf is denoted by $|\Ber|_M$. We denote its global sections by $|\Ber|M$.
\end{definition}

In the literature, Berezin forms are more common than densities. In keep with this convention, we will use forms in \autoref{stokes}. However, in general, it will be more convenient to work with densities; the extension to forms will always be straightforward. In particular, this applies to the formulation and proof of our main result, in \autoref{boundary_smfd}.

We recall the definition of the Berezin integral. 

\begin{definition}\label{def:3.3}
	Let $U$ be a coordinate neighbourhood with a coordinate system $x=(u,\xi)$, and $\omega=f |Dx|\in|\Ber|U$. We define 
	\begin{align}
		\int_\lr(U,x)\omega\defi(-1)^{s(p,q)}\int_{U_0}f_\lr(1,\dotsc,1)|du_0|,
	\end{align}
	whenever the right hand side exists. Here, $|du_0|$ is the pullback of the standard Lebesgue density on $\RR^p$ under $u_0$, and $f_\lr(1,\dotsc,1)$ is the top degree coefficient in \Cref{eq:1}, where $\gamma$ is associated with $u$. 

	There is no uniform choice for the number $s(p,q)\in\ZZ_2$ in the literature. Customary are $s(p,q)=pq+\fr {q(q-1)}/2$ or $s(p,q)=\fr {q(q-1)}/2$. The definition of the integral of a Berezin form is similar. 
\end{definition}

We have the following classical theorem \cite[Theorem~2.4.5]{Le80}.
\begin{theorem}\label{thm:3.4}
	Let $U$ be a coordinate neighbourhood and $\omega$ be a Berezin density which is compactly supported on $U$. Then
	\begin{align*}
		\int_\lr(U,x)\omega=\int_\lr(U,y)\omega,
	\end{align*}
	if $x=(u,\xi)$ and $y=(v,\eta)$ are coordinate systems on $U$. The same is true for Berezin forms if $u_0$ and $v_0$ are equally oriented.
\end{theorem}

As is well-known, the assumption of compact supports cannot be removed in the above theorem; the following classical counterexample is referred to as Rudakov's example in the literature.
\begin{example}\label{ex:3.4}
	Let $\Omega\subseteq\RR^{1,2}$ be the superdomain with $\Omega_0=\ ]0,1[$. Let $x=(u,\xi_1,\xi_2)$ be a coordinate system on $\Omega$ with $u_0=\id_{\Omega_0}$. Let $y=(v,\eta_1,\eta_2)$ be the coordinate system given by $v=u+\xi_1\xi_2$ and $\eta_i=\xi_i,\ i=1,2$. Set $\omega\defi vDy$. 
	
	We have
	\begin{align*}
		\Da y/x
			=\Ber\left(\begin{smallmatrix}
					1		& 0	& 0\\
					\xi_2	& 1	& 0\\
					-\xi_1				& 0	& 1
					\end{smallmatrix}\right)
			=1,
	\end{align*}
	hence $\omega=(u+\xi_1\xi_2)Dx$. This leads to
	\begin{align*}
		\int_\lr(\Omega,x)\omega=\pm1\neq0=\int_\lr(\Omega,y)\omega.
	\end{align*}
\end{example}

However, \autoref{thm:3.4} allows us to make the following observation.

\begin{lemma}\label{lem:3.5}
	Let $\gamma$ be a retraction on $M$ and $\omega$ a Berezin density. Let $x=(u,\xi)$ and $y=(v,\eta)$ be coordinate systems on a coordinate neighbourhood $U$ with the same associated retraction $\gamma$. Let $\omega=f|Dx|=g|Dy|$ on $U$. Then
	\begin{align*}
		f_\lr(1,\dotsc,1)|du_0|=g_\lr(1,\dotsc,1)|dv_0|,
	\end{align*}
	where $f_\lr(1,\dotsc,1)$ and $g_\lr(1,\dotsc,1)$ are the coefficients from \Cref{eq:1}, applied to $f$ and $g$, respectively. 
\end{lemma}
\begin{proof}
	Choose a bump function $h\in C_c^\infty(U_0)$. Then by \autoref{thm:3.4},
	\begin{align*}
		\int_{U_0}hf_\lr(1,\dotsc,1)|du_0|
			&=\pm\int_\lr(U,x)\gamma^*(hf_\lr(1,\dotsc,1))\xi_1\dotsm\xi_q\,|Dx|=\pm\int_\lr(U,x)\gamma^*(h)\omega\\
			&=\pm\int_\lr(U,y)\gamma^*(h)\omega=\int_{U_0}hg_\lr(1,\dotsc,1)|dv_0|.
	\end{align*}
	Since $h$ was arbitrary, this proves our claim. 
\end{proof}

Again, one can get the same result for Berezin forms. Thanks to this lemma, the following definition makes sense.

\begin{definition}\label{def:3.6}
	Let $\gamma$ be a retraction on the supermanifold $M$. We define the map $\gamma_!\colon|\Ber|M\to|\Omega^p|M_0$ locally \via
	\begin{align*}
		(\gamma_!\omega)|_U&\defi (-1)^{s(p,q)} f_\lr(1,\dotsc,1)|du_0|,
	\intertext{where $\omega$, $U$, $x=(u,\xi)$ and $f_\lr(1,\dotsc,1)$ are as in \autoref{lem:3.5}.
	\newline\indent Similar we define $\gamma_!\colon\Ber M\to\Omega^p M_0$ \via}
		\gamma_!(fDx)&\defi (-1)^{s(p,q)} f_\lr(1,\dotsc,1)du_0.
	\end{align*}

\end{definition}

In fact, $\varphi_!$ can be defined for any surjective submersion $\varphi$ \cite{AH10}. Note that if one chooses $b(p,q)=p+q$ or $b(p,q)=q$ and fixes parity according to the sign rule, the morphism $\gamma_!$ becomes even.

One can easily check the following properties:
\begin{align}\label{eq:3}
	\gamma_!\bg(\gamma^*(g)\omega)=g\gamma_!(\omega),\quad\supp\gamma_!(\omega)\subseteq\supp\omega
\end{align}
for any $g\in C^\infty(M_0)$ and $\omega\in|\Ber| M$ (resp.~$\omega\in\Ber M$).

\begin{definition}\label{def:3.7}
Let $\gamma$ be a retraction on $M$ and $\omega$ be a Berezin density on $M$. We call $\omega$ integrable with respect to $\gamma$ if $\gamma_!(\omega)$ is integrable on $M_0$ as density. In this case, we define
	\begin{align*}
		\int_\lr(M,\gamma)\omega\defi\int_{M_0}\gamma_!\omega.
	\end{align*}
	If $M_0$ is oriented, this definition can be extended to the case of Berezin forms.
\end{definition}

On coordinate neighbourhoods $U$ this definition is compatible with the local definition, given in \autoref{def:3.3}:
\begin{align*}
	\int_\lr(U,x)\omega=\int_\lr(U,\gamma)\omega,
\end{align*}
where $\gamma$ is the retraction associated with $x$.
In particular, the integral on the right hand side is the same for coordinate systems whose even parts induce the same retraction. Moreover, \autoref{thm:3.4} generalises as follows.
\begin{corollary}\label{retrcpt}
	Let $\gamma,\gamma'$ be retractions on $M$ and $\omega$ be compactly supported on $M$. Then
	\begin{align*}
		\int_\lr(M,\gamma)\omega=\int_\lr(M,\gamma')\omega.
	\end{align*}
\end{corollary}

In this case, we will write $\int_M\omega$ for the integral.

\begin{corollary}
	Let $\omega\in|\Ber| M$ (resp.~$\omega\in\Ber M$) and $\gamma,\gamma'$ be retractions. The density (resp.~volume form) $\gamma_!(\omega)-\gamma'_!(\omega)$ is exact. 
\end{corollary}

\begin{proof}
	If $\omega$ is compactly supported, then 
	\[
		\int_{M_0}\bigl(\gamma_!(\omega)-\gamma'_!(\omega)\bigr)=\int_{(M,\gamma)}\omega-\int_{(M,\gamma')}\omega=0,
	\]
	so $\gamma_!(\omega)-\gamma'_!(\omega)$ is exact.
	
	In the general case, let $\lr(\phi_\alpha)\subseteq\mathcal O(M)$ be a partition of unity with compact supports. By the above, $d\eta_\alpha=\gamma_!(\phi_\alpha\omega)-\gamma'_!(\phi_\alpha\omega)$ for some $\eta_\alpha$. One may assume the family of supports to be locally finite, so that $\eta=\sum_\alpha\eta_\alpha$ is well-defined, and one has $d\eta=\gamma_!(\omega)-\gamma'_!(\omega)$. 
\end{proof}

\begin{definition}
	Let $\varphi\colon M\to N$ be an isomorphism of supermanifolds.
	\begin{enumerate}[(i)]
		\item
			The \emph{pullback} Berezin density $\varphi^*\omega$ of a Berezin density $\omega$ on $N$ is defined by writing $\omega|_U= f |Dx|$ on a coordinate neighbourhood $(U,x)$ on $N$ and setting
			\begin{align*}
				\lr(\varphi^*\omega)|_{\varphi^{-1}(U)}\defi\varphi^*(f)|D\varphi^*(x)|.
			\end{align*}
			Here, we observe that $\varphi^*(x)=\big(\varphi^*(x_1),\dotsc,\varphi^*(x_{p+q})\big)$ is a coordinate system on $\varphi^{-1}(U)\defi M|_{\varphi_0^{-1}(U_0)}$.
			
			This is well-defined, since
			\begin{align*}
				\varphi^*\lr(\Da x/y)=\Da{\varphi^*(x)}/{\varphi^*(y)}.
			\end{align*}
			The \emph{pullback} of a Berezin form is defined analogously.
		\item
			The \emph{pullback} $\varphi^*\gamma$ of a retraction $\gamma$ on $N$ is defined by
			\begin{align*}
				\varphi^*\gamma\defi\varphi_0^{-1}\circ\gamma\circ\varphi\colon M_0\longrightarrow M.
			\end{align*}
	\end{enumerate}
\end{definition}

\begin{corollary}\label{cor:3.9}
	Let $\varphi\colon M\to N$ be an isomorphism of supermanifolds. Let $\gamma$ be a retraction on $N$ and let $\omega$ be a Berezin density or Berezin form on $N$ which is integrable with respect to $\gamma$. Then $\varphi^*\omega$ is integrable with respect to $\varphi^*\gamma$ and
	\begin{align*}
		\int_\lr(M,\varphi^*\gamma)\varphi^*\omega=\int_\lr(N,\gamma)\omega.
	\end{align*}
	In the case of a Berezin form, $\varphi_0$ has in addition to be orientation preserving.
\end{corollary}
\begin{proof}
	We only have to check that
	\begin{align}\label{eq:pullback}
		\lr(\varphi^*\gamma)_{!}\lr(\varphi^*\omega)=\varphi_0^*\lr(\gamma_!\omega),
	\end{align}
	because then
	\begin{align*}
		\int_\lr(M,\varphi^*\gamma)\varphi^*\omega
			=\int_{M_0}\lr(\varphi^*\gamma)_{!}\lr(\varphi^*\omega)
			=\int_{M_0}\varphi_0^*\lr(\gamma_!\omega)
			=\int_{N_0}\gamma_!\omega
			=\int_\lr(N,\gamma)\omega.
	\end{align*}
	
	It suffices to check \Cref{eq:pullback} locally. So, we write $\omega=f|Dx|$ and $f=\sum_\nu\gamma^*(f_\nu)\xi^\nu$ for a coordinate system $x=(u,\xi)$ with which $\gamma$ is associated. Note that $\varphi^*\gamma$ is the retraction associated with $\varphi^*(x)$. We decompose $\varphi^*\omega$ with respect to this coordinate system:

	\begin{align*}
		\varphi^*\omega
			=\sum_\nu\varphi^*\bg(\gamma^*(f_\nu)\xi^\nu)|D\varphi^*(x)|
			=\sum_\nu\lr(\varphi^*\gamma)^*\bg(\varphi_0^*(f_\nu))\varphi^*(\xi)^\nu|D\varphi^*(x)|.
	\end{align*}
	It follows that
	\begin{align*}
		\lr(\varphi^*\gamma)_{!}\lr(\varphi^*\omega)
			&=(-1)^{s(p,q)}\varphi_0^*(f_\lr(1,\dotsc,1))|d\varphi_0^*(u_0)|
			=\varphi_0^*\lr(\gamma_!\omega).\qedhere
	\end{align*}
\end{proof}

\section{Stokes's theorem}\label{stokes}

\begin{definition}
	Recall \cite{BL77,Man97} that the sheaf $\Sigma^k_M$ of \emph{integral forms} of order $k\leq p$ is defined to be
	\begin{align*}
		\Sigma_M^k\defi\Ber_M\otimes_{\OO_M}S^{p-k}(\XX_M\Pi),
	\end{align*}
	where $S^k(\XX_M\Pi)$ denotes the $k$-th supersymmetric power of the sheaf of parity changed super derivations. We will abbreviate $\Sigma_M^k(M_0)$ by $\Sigma^kM$. In the following we restrict to the case $k=p-1$.
	
	The \emph{Cartan derivative} on $p-1$ integral forms is given by
	\begin{align*}
		d\colon\Sigma^{p-1}M\longrightarrow\Sigma^pM=\Ber M,\ \omega\otimes X\Pi\longmapsto(-1)^{|\omega||X\Pi|}\Lie_X\omega.
	\end{align*}
	
	Here, $\Lie_X$ is the Lie derivative on $\Ber M$, locally given by
	\begin{align*}
		\Lie_{g\der /{x_i}}(fDx)=(-1)^{|g||x_i|}\der/{x_i}(gf)Dx.
	\end{align*}
	This does not depend on the chosen coordinate system \cite[Lemma~2.4.6]{Le80}.
	
	For conceptual reasons we made here a choice for the sign which differs from Ref.~\cite{Man97}. (The sign there is given by $(-1)^{|\omega||X\Pi|+|X|}$.)
\end{definition}

\begin{remark}\label{rk:4.2}
	In the classical case $M=M_0$ integral forms and differential forms can be identified. For $k=p-1$ this identification is given by
	\begin{align*}
		\Psi\colon\Sigma^{p-1}M_0\longrightarrow\Omega^{p-1}M_0,\ \omega\otimes X\Pi\longmapsto (-1)^{|\omega|}\iota_X\omega,
	\end{align*}
	where $\iota_X$ is the contraction by $X$. The definition of the Cartan derivative is compatible with this identification, as can be seen from 
	\begin{align}
		d\bg(\Psi(\omega\otimes X\Pi))=(-1)^{|\omega|}d(\iota_X\omega)=(-1)^{|\omega|}\Lie_X\omega.
	\end{align}
\end{remark}

\begin{definition}
	Recall that a morphism $\iota\colon N\to M$ is called an \emph{immersion} in case the following is true: For each point $o\in N_0$ and some (any) coordinate system $x=(x_1,\dotsc,x_{p+q})$ on a neighbourhood of $\iota_0(o)$, there exists a coordinate neighbourhood $U$ of $o$, such that $\big(\iota^*(x_{i_1}),\dotsc,\iota^*(x_{i_k})\big)$ is a coordinate system on $U$ for certain $i_1<\dotsb<i_k$.
\end{definition}

\begin{lemma}\label{la:4.3}
	If $\dim N=(p-k,q-l)$ and $\dim M=(p,q)$, one can choose $x=(u,\xi)$ such that $\iota^*(u_i)=\iota^*(\xi_j)=0$ for $i=1,\dotsc,k$ and $j=1,\dotsc,l$.
\end{lemma}

For the remainder of this section we suppose $N$ to be of dimension $(p-1,q)$ and $\iota\colon N\to M$ to be an immersion.

\begin{definition}\label{def:4.4}
	The \emph{pullback}
	\begin{align*}
		\iota^*\colon\Sigma^{p-1}M\longrightarrow\Sigma^{p-1}N=\Ber N
	\end{align*}
	of integral forms of order $p-1$ is defined  as follows: For each point $o\in M_0$, choose a coordinate system $x=(x_1,\tilde x)$ at $\iota_0(o)$ as in \autoref{la:4.3} and set 
	\begin{align}
		\iota^*\lr(fDx\otimes\der/{x_i}\Pi)
			\defi\begin{cases}
				(-1)^{|Dx|}\iota^*(f)D\iota^*(\tilde x)	& i=1,\\
				0											& i\neq1.
			 \end{cases}
	\end{align}
\end{definition}

\begin{remark}
	\autoref{def:4.4} is compatible with the classical pullback \via the identification $\Psi$ from \autoref{rk:4.2}. Let $u_0=(u_{1,0},\dotsc,u_{p,0})$ be as in \autoref{la:4.3} (\ie $\iota_0^*(u_{1,0})=0$). One computes 
	\begin{align*}
		\iota_0^*\bg(\Psi(fdu_0\otimes\der/{u_{i,0}}\Pi))
			&=\iota_0^*\big((-1)^{p+i+1}fdu_{1,0}\wedge\dotsb\wedge\widehat{du_{i,0}}\wedge\dotsb\wedge du_{p,0}\big)\\
			&=\begin{cases}
				(-1)^p\iota_0^*(f)d\iota_0^*(u_{2,0})\wedge\dotsb\wedge d\iota_0^*(u_{p,0})	& i=1,\\
				0																	& i\neq1.
			\end{cases}
	\end{align*}
\end{remark}

\begin{proposition}\label{prop:4.6}
	The definition of the pullback at a certain point does not depend on the choice of the coordinate system and hence, the pullback of integral forms of order $p-1$ is well-defined.
\end{proposition}
\begin{proof}
	Let $y=(y_1,\tilde y)$ be another such coordinate system with $\iota^*(y_1)=0$. We have to compute
	\begin{align*}
		\iota^*\lr(fDy\otimes\der/{y_i}\Pi)
			&=\iota^*\Bigg(\sum_{j=1}^{p+q}(-1)^{(|x_j|+|y_i|)|Dx|}f\der{x_j}/{y_i}\D y/xDx\otimes\der/{x_j}\Pi\Bigg)\\
			&=(-1)^{(|y_i|+1)|Dx|}\iota^*\lr(f\der{x_1}/{y_i}\D y/x)D\iota^*(\tilde x),
	\end{align*}
	for $i=1,\dotsc,p+q$. For $i\geq2$ we infer, using $\iota^*(y_1)=\iota^*(x_1)=0$,
	\begin{align*}
		0=\der\iota^*(x_1)/{\iota^*(y_i)}
			=\sum_{j=1}^{p+q}\der{\iota^*(y_j)}/{\iota^*(y_i)}\iota^*\bigg(\der{x_1}/{y_j}\bigg)
			=\sum_{j=2}^{p+q}\delta_{ij}\iota^*\bigg(\der{x_1}/{y_j}\bigg)
			=\iota^*\lr(\der{x_1}/{y_i}).
	\end{align*}
	
	This implies $\iota^*\lr(fDy\otimes\der/{y_i}\Pi)=0$ for $i\geq 2$. Now we examine 
	\begin{align*}
		\iota^*\lr(\D x/y)\!
			&=\Ber\!\left(\begin{smallmatrix}
						\iota^*\lr(\der {x_1}/{y_1})		&\cdots&	\iota^*\lr(\der {x_{p+q}}/{y_1})\\
						\vdots								&\ddots&	\vdots\\
						\iota^*\lr(\der {x_1}/{y_{p+q}})	&\cdots&	\iota^*\lr(\der {x_{p+q}}/{y_{p+q}})
					\end{smallmatrix}\right)\!
			=\Ber\!\left(\begin{smallmatrix}
						\iota^*\lr(\der {x_1}/{y_1})		& *											&\cdots&	*\\
						0										& \iota^*\lr(\der {x_2}/{y_2})		&\cdots&	\iota^*\lr(\der {x_{p+q}}/{y_2})\\
						\vdots									& \vdots									&\ddots&	\vdots\\
						0										& \iota^*\lr(\der {x_2}/{y_{p+q}})	&\cdots&	\iota^*\lr(\der {x_{p+q}}/{y_{p+q}})
					\end{smallmatrix}\right)\!\\
			&=\iota^*\lr(\der {x_1}/{y_1})\D{\iota^*(\tilde x)}/{\iota^*(\tilde y)}.
	\end{align*}
	
	Here, we have made use of
	\begin{align}\label{eq:6}
		\begin{split}
			\Ber\left(\begin{array}{cc|c}
								R_1 & *		& *\\
								0	& R_2	& S\\
								\hline
								0	& T		& V
							\end{array}\right)
				&=\det R_1\Ber\begin{pmatrix}
												R_2	& S\\
												T	& V
											\end{pmatrix}.
		\end{split}
	\end{align}
	
	We arrive at $\D \iota^*(\tilde y)/{\iota^*(\tilde x)}=\iota^*\big(\der x_1/{y_1}\big)\iota^*\big(\D y/x\big)$ by inverting both sides of the above equation; hence
	\begin{align*}
		\iota^*\lr(fDy\otimes\der/{y_1}\Pi)&=(-1)^{|Dy|}\iota^*(f)D\iota^*(\tilde y).\qedhere
	\end{align*}
\end{proof}

For the formulation of Stokes's theorem, we need to anticipate a later result (\autoref{cor:5.13}). Let $U\subset M$ such that $U_0$ has smooth boundary $\partial U_0$ in $M_0$. Further, let $\gamma$ be a retraction on $M$. 

Then there exists a \emph{unique} supermanifold structure $\partial_\gamma U$ of dimension $(p-1,q)$ on $\partial U_0$, together with an immersion $\iota\colon \partial_\gamma U\to M$ and a unique retraction $\partial\gamma$ on $\partial_\gamma U$ such that the following diagram commutes:
\begin{align}\label{boundarycomm}
	\begin{split}\begin{xy}
		\xymatrix{
			\partial_\gamma U\ar[r]^\iota\ar[d]_{\partial\gamma}		& M\ar[d]^\gamma\\
			\partial U_0\ar@{^{(}->}[r]_{\iota_0}	& M_0
		}
	\end{xy}\end{split}
\end{align}

\begin{theorem}[Stokes's theorem]\label{thm:stokes}\pdfbookmark[2]{Stokes's theorem for a fixed retraction}{stokes1}
	Let $U\subset M$ such that $\overline{U_0}$ is compact and has smooth boundary $\partial U_0$, and let $\gamma$ be a retraction on $M$. Let $M_0$ be oriented, and endow $\partial U_0$ with the usual boundary orientation. Then for $\varpi\in\Sigma^{p-1}M$ we have
	\begin{align}\label{eq:stokes}
		\int_\lr(U,\gamma)d\varpi=(-1)^{s(p,q)+s(p-1,q)+q}\int_{\partial_\gamma U}\iota^*(\varpi),
	\end{align}
	whenever the integral on the left hand side exists.
\end{theorem}

For the special choice $s(p,q)=pq+\fr{q(q-1)}/2$, the sign in Stokes's formula disappears. Therefore this choice might be reasonable in this context.

We make the following subtle point: The integral on the right hand side of \Cref{eq:stokes} does not depend on the boundary retraction $\partial\gamma$. However, one still has to take into account the boundary data $(\partial_\gamma U,\iota)$. In order to clarify this, we consider the following example.

\begin{example}\label{ex:4.6}
	Let $M=\RR^{1,4}$ and $U_0=\ ]0,\infty[$. Let $x=(u,\xi_1,\xi_2,\xi_3,\xi_4)$ be the standard coordinate system on $M$ ($u=u_0=\id_{M_0}$). We define another coordinate system $y=(v,\eta_1,\eta_2,\eta_3,\eta_4)$ by
	\begin{align*}
		v\defi u+\xi_1\xi_2+\xi_3\xi_4,\quad \eta_j=\xi_j,\ j=1,\dotsc, 4.
	\end{align*}
	
	Let $\gamma$ be the retraction associated with $y$, \ie $\gamma^*(v_0)=v$. In this example, there is only one possible supermanifold structure of dimension $(0,4)$ on $\partial U_0={0}$, namely $\partial U=\RR^{0,4}$. 
		
	Now one might think that the immersion $\iota$ is just given by
	\begin{align*}
		\iota^*\colon C^\infty(M_0)\otimes\Wedge\big(\RR^4\big)^{*}\longrightarrow\Wedge\big(\RR^4\big)^{*},\ \sum_\nu f_\nu\xi^\nu\longmapsto\sum_\nu f_\nu(0)\xi^\nu.
	\end{align*}
	
	Let us examine where this leads to. Define $\varpi\defi\fr 1/2 v^2Dy\otimes\der/v\Pi\in\Sigma^0 M$. We compute $d\varpi=\pm vDy$, which implies that $\int_\lr(U,\gamma)d\varpi=0$.
	
	Since $\iota^*(u)=0$ we have to calculate $\varpi$ in the $x$-coordinates. We see $\D y/x=1$ and $\der/v=\der/u$, hence $\varpi=\left(\fr 1/2 u^2+u(\xi_1\xi_2+\xi_3\xi_4)+\xi_1\xi_2\xi_3\xi_4\right)\otimes\der /u\Pi$. This means that $\iota^*(\varpi)=\pm\xi_1\xi_2\xi_3\xi_4D\xi$ and therefore
	\begin{align*}
		\int_{\partial U}\iota^*(\varpi)=\pm1\neq0=\pm\int_\lr(U,\gamma)d\varpi.
	\end{align*}
	
	The reason for this supposed contradiction is that with the chosen immersion $\iota$, Diagram~\eqref{boundarycomm} does not commute. The correct immersion is
	\begin{align*}
		\sum_\nu\gamma^*(f_\nu)\xi^\nu\longmapsto\sum_\nu f_\nu(0)\xi^\nu.
	\end{align*}
\end{example}

\begin{remark}
	Stokes's theorem for supermanifolds was proved \cite{BL77} for the case of domains with compact boundary. The domain of integration there is a \emph{closed superdomain} which is characterised locally by an equation $u_1\geq 0$. This corresponds to our choice of a retraction. The boundary of the closed superdomain is given locally by the equation $u_1=0$, similar to the unique structure on the boundary, which we get from diagram \eqref{boundarycomm}.
	
	In \cite{Man97}, the theorem is stated as follows: One starts with a supermanifold structure on the boundary together with an immersion. It is remarked that the boundary is given locally by an equation $u_1=0$ (\cf \autoref{la:4.3}), and $U$ by $u_1>0$. The conclusion as it is stated is correct only if the integral is evaluated by using a coordinate system which contains $u_1$. This means, that the integral of $U$ depends on the chosen immersion. 
	
	We feel that this formulation may easily be misunderstood, as in the above example, whereas the statement in terms of retractions might be more descriptive. As we shall see at the end of \autoref{boundary_smfd}, Stokes's theorem admits an extension to the case of an arbitrary immersion; however, in this case, additional terms will appear in the formula.
\end{remark}

In the \emph{proof} of \autoref{thm:stokes}, we need a generalisation of $\gamma_!$ to integral forms.
\begin{definition}
	Define $\gamma_!\colon\Sigma^{p-1}M\to\Sigma^{p-1}M_0=\Omega^{p-1}M_0$ locally \via
	\begin{align*}
		\gamma_!\lr(\omega\otimes\der/{x_i}\Pi)\defi
			\begin{cases}
				\gamma_!(\omega)\otimes\der/{u_{i,0}}\Pi	& i\leq p,\\
				0											& i>p.
			\end{cases}
	\end{align*}
	
	Here $x=(u,\xi)$ is a coordinate system with which $\gamma$ is associated. We check that the definition is independent of this choice. To that end, let $y=(v,\eta)$ be another coordinate system with $\gamma^*(v_0)=v$. 
	
	Then for $i=1,\dotsc,p$ we have $\der/{v_i}=\sum_{k=1}^p\gamma^*\big(\der{u_{k,0}}/{v_{i,0}}\big)\der/{u_k}$, hence
	\begin{align*}
		\gamma_!\left(\omega\otimes\der/{v_i}\Pi\right)
			&=\sum_k\gamma_!\left(\omega\,\gamma^*\lr(\der{u_{k,0}}/{v_{i,0}})\otimes\der/{u_k}\Pi\right)\\
			&=\sum_k\gamma_!(\omega)\der{u_{k,0}}/{v_{i,0}}\otimes\der/{u_{k,0}}\Pi
			=\gamma_!(\omega)\otimes\der/{v_{i,0}}\Pi.
	\end{align*}
	
	Similarly, we have $\gamma_!\big(\omega\otimes\der/{\eta_j}\Pi\big)=0$ for $j=1,\dotsc,q$. 
\end{definition}

\begin{proof}[\proofname\ of \autoref{thm:stokes}]
	We only need to check the equations
	\begin{align}
		\lr(\partial\gamma)_!\bg(\iota^*(\varpi))
			&=(-1)^{b(p,q)+b(p,0)+s(p,q)+s(p-1,q)}\,\iota_0^*\bg(\gamma_!(\varpi)),\label{eq:retpullb}\\
		\gamma_!(d\varpi)
			&=(-1)^{b(p,q)+b(p,0)+q}\,d\gamma_!(\varpi)\label{eq:retcartan}.
	\end{align}
	
	With these identities, we are able to apply the classical Stokes's theorem: 
	\begin{align*}
		\int_\lr(U,\gamma) d\varpi
			&=\int_{U_0}\gamma_!(d\varpi)
			=\pm\int_{U_0}d\gamma_!(\varpi)
			=\pm\int_{\partial U_0}\iota_0^*\bg(\gamma_!(\varpi))\\
			&=\pm\int_{\partial U_0}\lr(\partial\gamma)_!\bg(\iota^*(\varpi))
			=(-1)^{s(p,q)+s(p-1,q)+q}\int_\lr(\partial_\gamma U,\partial\gamma)\iota^*(\varpi).
	\end{align*}
	The claim follows from \autoref{retrcpt}.

	Equations \eqref{eq:retpullb} and \eqref{eq:retcartan} can be checked locally. Let $u_0=(u_{1,0},\dotsc,u_{p,0})$ be a coordinate system such that $\iota_0^*(u_{1,0})=u_{1,0}|_{\partial U_0}=0$ and set $u\defi\gamma^*(u_0)$. We supplement $u$ to a coordinate system $x=(u_1,\tilde x)=(u,\xi)$. 
	
	Without loss of generality, we write $\varpi=fDx\otimes\der/{x_i}\Pi$ and $f=\sum_\nu\gamma^*(f_\nu)\xi^\nu$. Noticing that $\partial\gamma$ is the retraction associated with $\iota^*(\tilde x)$, we get for $i=1$:
	\begin{align*}
		\lr(\partial\gamma)_!\bg(\iota^*(\varpi))
			&=(-1)^{|Dx|}\lr(\partial\gamma)_!\Big(\sum_\nu\iota^*\bg(\gamma^*(f_\nu))\iota^*(\xi)^\nu D\iota^*(\tilde x)\Big)\\
			&=(-1)^{|Dx|}\lr(\partial\gamma)_!\Big(\sum_\nu(\partial\gamma)^*\bg(\iota_0^*(f_\nu))\iota^*(\xi)^\nu D\iota^*(\tilde x)\Big)\\
			&=(-1)^{|Dx|+s(p-1,q)}\iota_0^*(f_\lr(1,\dotsc,1))d\iota_0^*(\tilde u_0)\\
			&=(-1)^{|Dx|+|du_0|+s(p-1,q)}\iota_0^*\left(f_\lr(1,\dotsc,1)du_0\otimes\der/{u_{1,0}}\Pi\right)\\
			&=(-1)^{b(p,q)+b(p,0)+s(p,q)+s(p-1,q)}\iota_0^*\bg(\gamma_!(\varpi)).
			\intertext{In case $i>1$, both sides of the equation vanish. As for the second equation, the case $i>p$ is easy, and we compute for $i\leq p$:}
		\gamma_!(d\varpi)
			&=\gamma_!\left((-1)^{|x_i\Pi||fDx|}\der f/{x_i}Dx\right)\\
			&=\gamma_!\left(\sum_\nu(-1)^{|\xi^\nu Dx|}\gamma^*\lr(\der {f_\nu}/{u_{i,0}})\xi^\nu Dx\right)\\
			&=(-1)^{|Dx|+q+s(p,q)}\,\der {f_\lr(1,\dotsc,1)}/{u_{i,0}}du_0\\
			&=(-1)^{|Dx|+|du_0|+q+s(p,q)}\,d\left(f_\lr(1,\dotsc,1)du_0\otimes\der/{u_{i,0}}\Pi\right)\\
			&=(-1)^{b(p,q)+b(p,0)+q}\,d\gamma_!(\varpi).\qedhere
	\end{align*}
\end{proof}

\section{Boundary terms---the local picture}\label{coordchange}

We will begin our examination of the behaviour of the Berezin integral under coordinate changes. In view of \autoref{cor:3.9}, what we need to understand is how the integrals for different retractions are related. 

We start with the following observation on a coordinate neighbourhood $U$. Let $\gamma$ and $\gamma'$ be retractions on $U$. Choose a classical coordinate system $u_0$ on $U_0$ and define $u\defi\gamma^*(u_0)$ and $v\defi\gamma'^*(u_0)$. We complete these to coordinate systems $x=(u,\xi)$ and $y=(v,\eta)$ with $\xi=\eta$.

Following \autoref{prop:3.3}, we know of the existence of a unique isomorphism $\varphi\colon U\to U$ such that $\varphi^*(x_i)=y_i$, $i=1,\dotsc,p+q$. Of course, this implies $\varphi_0=\id_{U_0}$ and $\varphi^*\gamma=\gamma'$.

If $\omega$ is a Berezin density on $U$, \autoref{cor:3.9} tells us that
\begin{align*}
	\int_\lr(U,\gamma')\varphi^*\omega=\int_\lr(U,\gamma)\omega,
\end{align*}
whenever one of both integrals exists. One might interpret this as a first formula for coordinate changes. However, a more explicit expression is desirable. For this reason we take a closer look at $\varphi$.

As one can conclude from the proof of \autoref{prop:3.3}, $\varphi$ is given by
\begin{align}\label{eq:phiretchange}
	\varphi^*=\sum_{j\in\multi(p,0)}\fr 1/{j!}(v-u)^i\diff u^i,
\end{align}
with $(v-u)^i\defi(v_1-u_1)^{i_1}\dotsm(v_p-u_p)^{i_p}$ and $i!\defi i_1!\dotsm i_p!$. Since $u_0=v_0$, we have that $v_s-u_s$ is nilpotent for each $s$, so the sum is finite. Thus, $\varphi^*$ is a differential operator of order at most $\lfloor\fr q/2\rfloor$.

There is a natural action of differential operators on Berezin densities; the following proposition can be found in Ref.~\cite{Ch94}.

\begin{proposition}
	Let $\Diff(M)$ be the set of differential operators on $M$. There is a unique $\OO(M)$-right linear action of $\Diff(M)$ on $|\Ber| M$ such that
	\begin{align}\label{eq:diffactber}
		\omega.X=-(-1)^{|X||\omega|}\Lie_X\omega
	\end{align}
	for all $X\in\XX M\subseteq\Diff(M)$ and $\omega\in|\Ber|M$.
\end{proposition}

The corresponding statement for $\Ber M$ is also correct. Note that the additional minus sign in \Cref{eq:diffactber} cannot be omitted. 

The so defined action is compatible with restrictions and pullbacks, \ie
\begin{align}
	(\omega.A)|_U		&=\omega|_U.A|_U,\\
	\varphi^*(\omega.A)	&=\varphi^*(\omega).\varphi^*(A)\label{eq:13},
\end{align}
where $\varphi^*(A)\defi\varphi^*\circ A\circ\varphi^{*-1}$.

In the local picture this action has the form:
\begin{align}\label{eq:diffactberexplicit}
	\omega.A=|Dx|\sum_{j\in\multi(p,q)}(-1)^{|j|+|fa_j||j_\odd|+\fr{|j_\odd|(|j_\odd|-1)}/2}\diff x^j(fa_j).
\end{align}

Here, $\omega=|Dx|f$ and $A=\sum_ja_j\diff x^j$; moreover, we set 
\begin{align*}
	j_\odd\defi(j_{p+1},\dotsc,j_{p+q})\quad\text{and}\quad|j|\defi j_1+\dotsb+j_{p+q}.
\end{align*}

With this definition, the pullback \via a morphism and the action of differential operators are compatible. 

\begin{proposition}\label{thm:5.2}
	Let $\varphi\colon M\to M$ be a isomorphism such that $\varphi^*$ is a differential operator. Then we have for each Berezin density $\omega\in|\Ber| M$
	\begin{align*}
		\omega=\varphi^*(\omega.\varphi^*).
	\end{align*}
	The same is true for Berezin forms, if $\varphi_0$ is orientation preserving.
\end{proposition}

\begin{proof}
	Let $h\in\OO(M)$ be compactly supported. In our notation, integration by parts takes the form
	\begin{align}\label{eq:partialint}
		\int_M(\omega.X)h=\int_M\omega X(h)
	\end{align}
	for any derivation $X\in\XX M$. To see this, one checks
	\begin{align*}
		\omega X(h)-(\omega.X) h=\omega X(h)+(-1)^{|\omega||X|}(\Lie_X\omega)h=(-1)^{|\omega||X|}\Lie_X(\omega h).
	\end{align*}
	
	Since $\omega h$ is compactly supported, the integral of the right hand side vanishes \cite[Lemma~2.4.8]{Le80}.
	
	Iteratively applying \Cref{eq:partialint}, we get for any $A\in\Diff(M)$
	\begin{align*}
		\int_M(\omega.A)h=\int_M\omega A(h).
	\end{align*}
	
	Using \autoref{cor:3.9}, we conclude
	\begin{align*}
		\int_M\varphi^*(\omega.\varphi^*)h\,
			&=\int_M(\omega.\varphi^*)\varphi^{*-1}(h)\,
			=\int_M\omega\,\varphi^*\bg(\varphi^{*-1}(h))
			=\int_M\omega h.
	\end{align*}
	
	Since $h$ was arbitrary, the assertion follows. 
\end{proof}

\begin{remark}
	\autoref{thm:5.2} is closely related to Ref.~\cite[Theorem~3.2]{Ro87}. It is shown there that for $\varphi^*=\sum_j a_j\diff x^j$, the inverse is 
	\begin{align}\label{eq:e-Y}
		\varphi^{*-1}=\sum_j(-1)^{|j|+\fr{|j_\odd|(|j_\odd|+1)}/2}\diff x^j\circ a_j\D y/x,
	\end{align}
	where $y=\varphi^*(x)$ for certain coordinate systems $x$ and $y$ (where $\varphi^*$ is denoted $e^Y$). 
	
	In fact, \Cref{eq:e-Y} can be deduced from \Cref{thm:5.2} and \Cref{eq:diffactberexplicit} by calculating for $f\in\OO(U)$
	\begin{align*}
		Dx\,\varphi^{*-1}(f)=\varphi^{*-1}(Dy f)=\lr(Dx\D y/xf).\varphi^*.
	\end{align*}
	
	As an aside, note that the coefficients are $a_j=\fr 1/{j!}(y_{p+q}-x_{p+q})^{i_{p+q}}\dotsm(y_1-x_1)^{i_1}$.
\end{remark}

We use \autoref{thm:5.2} to derive an explicit expression for the Berezin integral under the change of retractions. 

\begin{theorem}\label{thm:5.4}\pdfbookmark[2]{Formula for coordinate changes}{formula}
	Let $U$ be a coordinate neighbourhood with two retractions $\gamma$ and $\gamma'$. Let $u_0$ be a coordinate system on $U_0$ and set $u=\gamma^*(u_0)$. Then $\omega\in|\Ber|U$ (or $\omega\in\Ber U$) is integrable with respect to $\gamma'$ and
	\begin{align*}
		\int_\lr(U,\gamma')\omega
			=\int_\lr(U,\gamma)\omega+\sum_{i\neq0}\fr1/{i!}\int_\lr(U,\gamma)\omega\big(\gamma'^*(u_0)-\gamma^*(u_0)\big)^i.\diff u^i,
	\end{align*}
	if the right hand side exists. Here, the sum is finite, and extends over $i\in\multi(p,0)$.
\end{theorem}

\begin{proof}
	Let $x=(u,\xi)$, $y=(v,\eta)$ with $v=\gamma'^*(u_0)$ and $\xi=\eta$. With the morphism $\varphi^*$ from \Cref{eq:phiretchange} ($\varphi^*(x)=y$) we get
		\begin{align*}
		\int_\lr(U,\gamma')\omega
			&=\int_\lr(U,\gamma')\varphi^*(\omega.\varphi^*)
			=\int_\lr(U,\gamma)\omega.\varphi^*
			=\int_\lr(U,\gamma)\omega+\sum_{i\neq0}\fr1/{i!}\int_\lr(U,\gamma)\omega(v-u)^i.\diff u^i.\qedhere
	\end{align*}
\end{proof}

We assemble our results in a general change of variables formula. 

\begin{corollary}
	Suppose given coordinate systems $x=(u,\xi)$ and $y=(v,\eta)$ on $U$ and $f\in\OO(U)$. Let $\hat u_s=g_s(v)$ such that $u_s\equiv\hat u_s\text{ mod }\langle\eta\rangle$. Then
	\begin{align*}
		\int_\lr(U,y)f|Dy|=\int_\lr(U,x)f\Da y/x |Dx|+\sum_{i\neq0}\fr1/{i!}\int_\lr(U,x)\diff u^i\left(f(u-\hat u)^i\Da y/x \right) |Dx|,
	\end{align*}
	if the integrals on the right hand side exist.
\end{corollary}

Note that we eliminated the sign $(-1)^{|i|}$ by replacing $(\hat u-u)$ with $(u-\hat u)$. Further, $\hat u$ always exists: $\hat u=\gamma'^*(u_0)$, where $\gamma'$ is the retraction associated with $v$.

\begin{remark}
	Observe that the above results bear a similarity to Ref.~\cite[Theorem~3.2]{Ro87}. Compared to Rothstein's result, the advantage of our theorem is that it is formulated in terms of Berezin densities, rather than of volume form valued differential operators. Whereas the former are a locally free $\OO_M$-module of rank $(0,1)^q$, the latter form one of infinite rank. 
\end{remark}

We apply these considerations in a few examples.

\begin{examples}\label{ex:5.6}\mbox{}
	\begin{enumerate}[(i)]
		\item
			Recall the notation from Rudakov's example (\autoref{ex:3.4}). Here we have $u=v-\eta_1\eta_2$, hence $\hat u=v$. For $f=\sum_\nu\gamma^*(f_\nu)\xi^\nu\in\OO(\RR^{1,2})$ and $\gamma^*(u_0)=u$ our recipe shows that
			\begin{align*}
				\int_\lr(\Omega,y)f|Dy|
					&=\int_\lr(\Omega,x)f|Dx|+\int_\lr(\Omega,x)\der/u\bg(f(u-v))|Dx|\\
					&=\int_\lr(\Omega,x)f|Dx|+\int_\lr(\Omega,x)\der/u(-f\xi_1\xi_2)|Dx|\\
					&=\int_\lr(\Omega,x)f|Dx|-\int_\lr(\Omega,x)\gamma^*\lr(\der{f_0}/{u_0})\xi_1\xi_2|Dx|\\
					&=\int_\lr(\Omega,x)f|Dx|-(-1)^{s(1,2)}\int_0^1\der{f_0}/{u_0}|du_0|\\
					&=\int_\lr(\Omega,x)f|Dx|-(-1)^{s(1,2)}\big(f_0(1)-f_0(0)\big).
			\end{align*}
			
			Comparing to the computations in \autoref{ex:3.4} $(f=v)$, this resolves the apparent contradictions. 
		\item\label{ex:5.6.2}
			Suppose $\Omega\subset\RR^{2,2}$ with $\Omega_0=\{(o_1,o_2)\mid o_1^2+o_2^2<1\}$. Let $y=(v,\eta)$ be a coordinate system on $\Omega$ with $v_0=\id_{\Omega_0}$. We want to compute the $y$-related integral of a compactly supported $f\in\OO(\Omega)$, by using rotational symmetry. Thus, we consider on $\Omega'\subset\Omega$ with $\Omega'_0=\Omega_0\backslash(]-\infty,0]\times 0)$ and a coordinate system $x=(u,\xi)$ on $\Omega'$, such that 
			\begin{align*}
				v_1		&=u_1\cos(u_2)(1-\xi_1\xi_2),	&\eta_1	&=u_1\xi_1,\\
				v_2		&=u_1\sin(u_2)(1-\xi_1\xi_2),	&\eta_2	&=u_1\xi_2.
			\end{align*}
			One computes $v_1^2+v_2^2+2\eta_1\eta_2=u_1^2$ and $\Da y/x=\fr1/{u_1}$. It remains to find $\hat u$.
			
			We have
			\begin{align*}
				u_1=\sqrt{v_1^2+v_2^2+2\eta_1\eta_2}
					=\sqrt{v_1^2+v_2^2}+\fr{\eta_1\eta_2}/{\sqrt{v_1^2+v_2^2}}\,,
			\end{align*}
			hence $\hat u_1=\sqrt{v_1^2+v_2^2}$ and
			\begin{align*}
				u_1-\hat u_1=\fr{\eta_1\eta_2}/{\sqrt{v_1^2+v_2^2}}
					=\fr{u_1^2\xi_1\xi_2}/{\sqrt{u_1^2(1-2\xi_1\xi_2)}}
					=u_1\xi_1\xi_2.
			\end{align*}
			Furthermore, we realise $\fr{v_1}/{v_2}=\tan(u_2)$, hence $\hat u_2=u_2$. This means that the second boundary term will vanish. 
			
			We write a compactly supported $f\in\OO(\Omega)$ as $f=\sum_\nu\gamma^*(f_\nu)\xi^\nu$ on $\Omega'$, where $\gamma$ is associated with $u$. We obtain
			\begin{align*}
				\int_\lr(\Omega,y)f|Dy|
					&=\int_\lr(\Omega',x)f\fr1/{u_1}|Dx|+\int_\lr(\Omega',x)\der/{u_1}\Big(fu_1\xi_1\xi_2\fr1/{u_1}\Bigr)|Dx|\\
					&=\int_\lr(\Omega',x)f\fr1/{u_1}|Dx|+\int_\lr(\Omega',x)\gamma^*\lr(\der{f_0}/{u_{1,0}})\xi_1\xi_2 |Dx|\\
					&=\int_\lr(\Omega',x)f\fr1/{u_1}|Dx|+(-1)^{s(2,2)}\int_{-\pi}^\pi\int_0^1\der{f_0}/{u_{1,0}}du_{1,0}du_{2,0}\\
					&=\int_\lr(\Omega',x)f\fr1/{u_1}|Dx|-(-1)^{s(2,2)}2\pi f_0(0).
			\end{align*}
			
			A similar computation is contained in \cite[proof of Theorem~1]{Zir91}.
		\item\label{ex:5.6.3}
			We redo the calculation in the first example in a more general context. Suppose $\Omega,\Omega'\subset\RR^{2,4}$, such that $\Omega_0=\Omega'_0\cap\RR_+^2$, where $\RR_+\!=\ ]0,\infty[$. Let $\gamma$ and $\gamma'$ be retractions on $\Omega'$ and let $u_0=\id_{\Omega'_0}$ be the standard coordinate system on $\Omega'_0$. Then \autoref{thm:5.4} tells us for a compactly supported Berezin density $\omega\in|\Ber|\Omega'$
			\begin{align*}
				\int_\lr(\Omega,\gamma')\omega
					=\int_\lr(\Omega,\gamma)\omega+\sum_{i\neq0}\int_\lr(\Omega,\gamma)\omega_i.\diff u^i,
			\end{align*}
			where $u\defi\gamma^*(u_0)$ and
			\begin{align*}
				\omega_i\defi\fr1/{i!}\omega\big(\gamma'^*(u_0)-\gamma^*(u_0)\big)^i=\sum_\nu\gamma^*(f_\nu^i)\xi^\nu|Dx|.
			\end{align*}
			Therefore,
			\begin{align*}
					\int_\lr(\Omega,\gamma)\omega_i.\diff u^i
						=\,(-1)^{s(p,q)+|i|}\int_0^\infty\int_0^\infty\diff {u_0}^i f_\lr(1,\dotsc,1)^i(o)do_1do_2.
			\end{align*}
			Applying the Fundamental Theorem of Calculus, we get 
			\begin{align*}
				\int_\lr(\Omega,\gamma')\omega
					=\int_\lr(\Omega,\gamma)\omega\,
					\pm\bigg(
						&-\int_0^\infty\left(f_\lr(1,\dotsc,1)^\lr(1,0)(0,o_2)-\derf f_\lr(1,\dotsc,1)^\lr(2,0)/2(0,o_2)\right)do_2\\
						&-\int_0^\infty\left(f_\lr(1,\dotsc,1)^\lr(0,1)(o_1,0)-\derf f_\lr(1,\dotsc,1)^\lr(0,2)/1(o_1,0)\right)do_1\\
						&+\derf{f_\lr(1,\dotsc,1)^\lr(1,1)(o)}/{1,2}(0,0)
						\bigg).
			\end{align*}
			This shows explicitly that the `boundary terms' indeed depend only on the values of $\omega$, and its derivatives, on the boundary. We shall presently exploit this to derive a global expression for the boundary terms. 
	\end{enumerate}
\end{examples}

\section{Boundary terms---the global picture}\label{boundary_smfd}

In this section, we will globalise the results of the previous section, using ideas from \autoref{ex:5.6}\,\eqref{ex:5.6.3}. A framework which is well suited to such a generalisation is that of supermanifolds built over manifolds with corners \cite{Mel93,Mel96}. Locally, such spaces are modelled on $\RR_+^k\!\times\RR^{p-k}$.

To exclude strange example such as the drop (\cf \autoref{fig:1}), we introduce the concept of boundary functions, \cf Ref.~\cite[Chapter 2]{Mel96}. 
\begin{figure}[ht]
	\centering
	\includegraphics[scale=0.9]{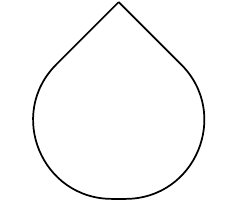}	
	\caption{The drop, which is not a manifold with corners (see below)}\label{fig:1}
\end{figure}

As before, $M$ will denote a supermanifold of dimension $(p,q)$ and $M_0$ will be the underlying manifold.

\begin{definition}
	A family of smooth functions $(\rho_1,\dotsc,\rho_r)$ is called \emph{independent at $o\in M_0$}, if the Jacobian $J_\lr(\rho_1,\dotsc,\rho_r)(o)$ at $o$ is of full rank.
	
	A family $(\rho_1,\dotsc,\rho_n)$ is called a family of \emph{boundary functions}, if the $\rho_i$ are independent at each point at which they vanish. This means for each subfamily $(\rho_{i_1},\dotsc,\rho_{i_k})$ and every $o\in M$:
	\begin{align*}
		\rho_{i_s}(o)=0\text{ for }s=1,\dotsc,k\Longrightarrow (\rho_{i_1},\dotsc,\rho_{i_k})\text{ is independent at }o.
	\end{align*}
\end{definition}

This implies that at most $p$ boundary functions can vanish simultaneously. Note that $n$ does not have to be smaller than $p$ (think of the case of an interval or a rectangle). Observe also that $(\rho_{i_1},\dotsc,\rho_{i_k})$ being independent at $o$ implies that this family can be supplemented to a coordinate system $(\rho_{i_1},\dotsc,\rho_{i_k},f_{k+1},\dotsc,f_p)$ on sufficiently small neighbourhoods of $o$. 

\begin{definition}
	A subset $N_0\subset M_0$ is called a \emph{manifold with corners}, if there exist boundary functions $\rho=(\rho_1,\dotsc,\rho_n)$ such that
	\begin{align*}
		N_0=&\,\Big\{o\in M_0\,\Big|\,\rho_i(o)>0,\ i=1,\dotsc,n\Big\}.
	\intertext{For each subfamily $\rho'=(\rho_{i_1},\dotsc,\rho_{i_k})$ we consider the set}
		H_0\defi&\left\{o\in M_0\ \middle|
				\begin{array}{ll}
					\rho_i(o)=0	& \text{for }\rho_i\in\rho',\\
					\rho_i(o)>0	& \text{for }\rho_i\notin\rho'
				\end{array}
			\right\}.
	\end{align*}
	
	Whenever $H_0$ is non-empty, it is called a \emph{boundary manifold} of $N_0$ of codimension $k$. We set $\rho_{H_0}\defi\rho'$ and denote by $B_0(M_0,\rho)=B_0(\rho)$ the collection of all boundary manifolds. Each boundary manifold of $N_0$ is a submanifold of $M_0$. Furthermore, the disjoint union of all boundary manifolds coincides with the (topological) boundary of $N_0$ in $M_0$.
	
	For later uses, we define for each boundary manifold $H_0\in B_0(\rho)$ a set of multi-indices
	\begin{align*}
		J_{H_0}\defi J_{H_0}^\rho\defi\Big\{j\in\multi(n,0)\,\Big|\,j_i=0\iff\rho_i\notin\rho_{H_0}\Big\}.
	\end{align*}
	
	Note that $\multi(n,0)\setminus\{0\}$ is the disjoint union of the $J_{H_0}$. Observe that for $j\in J_{H_0}$, the function $\rho^j=\rho_1^{j_1}\dotsm\rho_n^{j_n}$ is a monomial in the boundary functions $\rho_{H_0}$.
\end{definition}

\begin{examples}\mbox{}
	\begin{enumerate}[(i)]
		\item The drop $N$ depicted in \autoref{fig:1} is \emph{not} a manifold with corners. Indeed, suppose the contrary. Since there is only one codimension one boundary manifold, there can only be one boundary function. But then $N$ has only one boundary manifold, and is a manifold with boundary, contradiction. 
		\item
			Easy examples of manifolds with corners are displayed in \autoref{fig:2}.
			\begin{figure}[ht]
				\centering
				\includegraphics[scale=0.8]{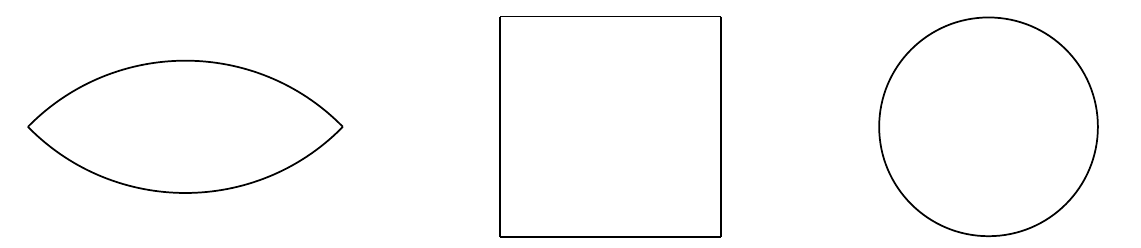}
				\caption{Manifolds with corners}\label{fig:2}
			\end{figure}
		\item
			One can also consider $N_0=\RR_+^k\!\times\RR^{p-k}$ (with $\RR_+\!=\ ]0,\infty[$) as a manifold with corners in $M_0=\RR^p$, with boundary functions $\rho=(\pr_1,\dotsc,\pr_k)$. These are the model spaces for manifolds with corners.
	\end{enumerate}
\end{examples}

We generalise this definition to the setting of supermanifolds.

\begin{definition}\label{def:5.10}\mbox{}
	\begin{enumerate}[(i)]
		\item
			A family of even superfunctions $\tau=(\tau_1,\dotsc,\tau_n)$ on $M$ is called a family of \emph{boundary superfunctions} if the family of underlying functions $\tau_0=(\tau_{1,0},\dotsc,\tau_{n,0})$ is a family of boundary functions.
		\item
			An open subspace $N\subset M$ is called a \emph{supermanifold with corners} if there exist boundary superfunctions $\tau=(\tau_1,\dotsc,\tau_n)$ such that $N_0\subset M_0$ is a manifold with corners \via the boundary functions $\tau_0$.
	\end{enumerate}
\end{definition}

\begin{remarks}
	Let $o\in M_0$ and $\tau'$ be a subfamily of $\tau$ such that $\tau_{i,0}(o)=0$ for each $\tau_i\in\tau'$. Similarly to the purely even case, one can augment the family $\tau'$ to a coordinate system $x=(u,\xi)=(\tau',\tilde x)=\bg((\tau',\tilde u),\xi)$ on a sufficiently small neighbourhood of $o$.

	If $N\subset M$ is given such that $N_0$ is a manifold with corners with boundary functions $\rho$, we are always able to turn $N$ into to a supermanifold with corners \via the boundary superfunctions $\tau=\gamma^*(\rho)$ for a retraction $\gamma$ on $M$. 
\end{remarks}

For the remainder of this section, $N$ will be a supermanifold with corners contained in the supermanifold $M$; moreover, $\tau=(\tau_1,\dotsc,\tau_n)$  will be boundary superfunctions, and $\rho=(\rho_1,\dotsc,\rho_n)$ will be boundary functions.

\begin{proposition}\label{prop:5.11}
	Let $\tau$ be given and $H_0\in B_0(\tau_0)$. Then there exist
	\begin{itemize}
		\item a supermanifold $H$ of dimension $(\dim H_0,q)$ with underlying space $H_0$,
		\item an immersion $\iota_H\colon H\to M$ over the inclusion $\lr(\iota_H)_0=\iota_{H_0}\colon H_0\hookrightarrow M_0$ such that $\iota_H^*(\tau_i)=0$ whenever $\tau_{i,0}\in\rho_{H_0}$.
	\end{itemize}
	The data $(H,\iota_H)$ are determined uniquely up to unique isomorphism by these conditions.
\end{proposition}

\begin{proof}This follows from \cite[Propositions~3.2.6]{Le80}.\end{proof}

Concretely speaking, the condition $\iota_H^*(\tau_i)=0$ means that $H$ is the supermanifold obtained by setting the boundary coordinates to zero.

\begin{definition}
	We call the supermanifolds $H$ from \autoref{prop:5.11} the \emph{boundary supermanifolds} of $N$ corresponding to $\tau$. The set of all such $H$ is denoted by $B(M,\tau)=B(\tau)$. We define abbreviations $\tau_H\defi\lr(\tau_i)_{\tau_{i,0}\in\tau_{H_0}}\!$ and $J_H^\tau\defi J_H^{\tau_0}\defi J_H\defi J_{H_0}$.
\end{definition}

We will need to integrate Berezin densities on $M$ along the boundary supermanifolds $H$. For that purpose, we define on $H$ canonical retractions, as well as the restriction to $H$ of Berezin densities on $M$.

\begin{lemma}\label{la:5.12}
	Let $\gamma$ be a retraction on $M$ and $H\in B\bg(\gamma^*(\rho))$. Then there exists a unique retraction $\gamma_H$ on $H$ such that the following diagram commutes:
	\begin{align}\label{eq:17}
		\begin{split}\begin{xy}
		\xymatrix{
			H\ar[r]^{\iota_H}\ar[d]_{\gamma_H}		& M\ar[d]^\gamma\\
			H_0\ar@{^{(}->}[r]_{\iota_{H_0}}	& M_0
		}
		\end{xy}\end{split}
	\end{align}
\end{lemma}

\begin{proof}
	 The uniqueness of $\gamma_H$ follows directly from $\iota_{H_0}\circ\gamma_H=\gamma\circ\iota_H$, since $\iota_{H_0}$ is a monomorphism. 
	 
	 For the existence, we choose $o\in H_0$ and complete the boundary functions $\rho_{H_0}$ to a coordinate system $u_0=(\rho_{H_0},\tilde u_0)$ on a sufficiently small neighbourhood $U_0\subset M_0$ of $o$. Since $\iota_{H_0}^*(\rho_{H_0})=\rho_{H_0}|_{H_0}=0$, it is clear that $\iota_{H_0}^*(\tilde u_0)=\tilde u_0|_{U_0\cap H_0}$ is a coordinate system on $U_0\cap H_0$. We set
	 \begin{align*}
		\gamma_H^*\lr(\tilde u_0|_{U_0\cap H_0})\defi\iota_H^*\bg(\gamma^*(\tilde u_0)).
	 \end{align*}
	 
	By \autoref{prop:3.3}, this defines a morphism $H|_{U_0\cap H_0}\to U_0\cap H_0$ such that the diagram on $H|_{U_0\cap H_0}$ corresponding to \eqref{eq:17} commutes. By uniqueness, this definition does not depend on the chosen coordinate system. Therefore, we can glue these morphisms to $\gamma_H\colon H\to H_0$. 
\end{proof}

\begin{proposition}\label{cor:5.13}
	Let $H_0\in B_0(\rho)$ be a boundary manifold and suppose that there is another family of boundary functions $\rho'$ which also determines $H_0$. Then the families of boundary superfunctions given by $\gamma^*(\rho)$ and $\gamma^*(\rho')$ determine the same supermanifold structure $H$ over $H_0$ with immersion $\iota_H$.
\end{proposition}

\begin{proof}
	Let $H\in B\bg(\gamma^*(\rho))$ be the supermanifold over $H_0$ associated with $\gamma^*(\rho)$. By the uniqueness in \autoref{prop:5.11}, it suffices to show $\iota_H^*\bg(\gamma^*(\rho'_{H_0}))=0$. Since $\iota_{H_0}^*(\rho'_{H_0})=\rho'_{H_0}|_{H_0}=0$, \autoref{la:5.12} proves the claim.
\end{proof}

\autoref{cor:5.13} shows that Diagram \eqref{eq:17} uniquely determines the supermanifold $H$ for a given retraction $\gamma$ on $M$. We have already made use of this in the statement of Stokes's theorem.

\begin{definition}\label{def:5.15}
	Let $H\in B(\tau)$ be a boundary supermanifold. We define a restriction map
	\begin{align*}
		|\Ber| M\longrightarrow|\Ber| H,\ \omega\longmapsto\omega|_{H,\tau}
	\end{align*}
	as follows. For $o\in H_0$ supplement $\tau_H$ to a coordinate system $x=(\tau_H,\tilde x)$ on a neighbourhood $U$ and write $\omega|_U=f|Dx|$. Since $\iota_H^*(\tau_H)=0$, the family $\iota_H^*(\tilde x)$ is a coordinate system on $U\cap H\defi H|_{U_0\cap H_0}$. Therefore, we may define
	\begin{align*}
		\omega|_{U\cap H,\tau}\defi\iota_H^*(f)|D\iota_H^*(\tilde x)|.
	\end{align*}
	
	This definition depends on $\tau$, but it is independent of the choice of the $\tilde x$ (see below). Thus, this defines a morphism of sheaves $|\Ber|_M\to\iota_{H_0*}|\Ber|_H$. By definition, $(f\omega)|_{H,\tau}=\iota_H^*(f)\omega|_{H,\tau}$, so it is in fact a morphism of $\OO_M$-modules.
\end{definition}

\begin{lemma}
	The construction in \autoref{def:5.15} does not depend on the choice of $x$.
\end{lemma}

\begin{proof}
	This proof is similar to that of \autoref{prop:4.6}. Let $y=(\tau_H,\tilde y)$ be another such coordinate system and let $k$ be the codimension of $H_0$. Then for $x_l\in\tilde x$ we obtain by the chain rule
	\begin{align*}
		\iota_H^*\lr(\der {y_i}/{x_l})
			=\der {\iota_H^*(y_i)}/{\iota_H^*(x_l)}.
	\end{align*}
	
	Furthermore, $\der {y_i}/{x_l}=\delta_{il}$ for $y_i\in\tau_H$, hence
	\begin{align*}
		\iota_H^*\lr(\Da y/x)
			&=\pm\Ber\left(\begin{smallmatrix}
						\iota_H^*\lr(\der {y_1}/{x_1})		&\cdots&	\iota_H^*\lr(\der {y_{p+q}}/{x_1})\\
						\vdots									&\ddots&	\vdots\\
						\iota_H^*\lr(\der {y_1}/{x_{p+q}})	&\cdots&	\iota_H^*\lr(\der {y_{p+q}}/{x_{p+q}})
					\end{smallmatrix}\right)
			=\pm\Ber\begin{pmatrix}
						1_{k\times k}	& *													\\
						0						& \der{\iota_H^*(\tilde y)}/{\iota_H^*(\tilde x)}	\\
					\end{pmatrix}\\
			&=\Da{\iota_H^*(y)}/{\iota_H^*(x)},
	\end{align*}
	using \Cref{eq:6}. Now suppose a Berezin density $f|Dy|$. We finish with
	\begin{align*}
		\big(f|Dy|\big)\big|_{U\cap H,\tau}
			&=\iota_H^*\lr(f \Da y/x) |D\iota_H^*(\tilde x)|
			=\iota_H^*(f)\Da{\iota_H^*(y)}/{\iota_H^*(x)}|D\iota_H^*(\tilde x)|\\
			&=\iota_H^*(f)|D\iota_H^*(\tilde y)|.\qedhere
	\end{align*}
\end{proof}

\begin{remarks}\label{rk:5.16}\mbox{}
	\begin{enumerate}[(i)]
		\item
			The restriction can also be defined for Berezin forms; in this case, one has to fix an ordering on the family of boundary superfunctions $\tau$. 
			
			The restriction of Berezin forms and the pullback of integral forms are related as follows: Suppose $x=(u,\xi)$ is a coordinate system on a superdomain $\Omega$. Let $(H,\iota_H)$ be the boundary data given by the boundary function $u_1$. Then for any Berezin form $\omega\in\Ber\Omega$ we have
			\begin{align*}
				\omega|_{H,u_1}=\iota_H^*\bg((-1)^{|Dx|}\omega\otimes\der/{u_1}).
			\end{align*}
		\item\label{rk:5.16.2}
			The restriction of Berezin densities is compatible with the notion of Riemannian measure; we elaborate this in the even case. Let $M_0$ carry a Riemannian metric $g$. The induced Riemannian density $\omega_g$ locally has the form 
			\begin{align*}
				\omega_g=\sqrt{|\det(g_{kj}^{u_0})|}\,|du_0|,
			\end{align*}
			for a coordinate system $u_0$, where $g_{kj}^{u_0}(o)=g_o\big(\der/{u_{k,0}}|_o,\der/{u_{j,0}}|_o\big)$ for $k,j=1,\dotsc,p$ and $o\in M_0$.
			
			Consider $N_0\subset M_0$ to be a manifold with corners and enumerate the boundary manifolds of codimension $1$ by $H^1_0,\dotsc,H^n_0$. The metric on $M_0$ determines on each $H^i_0$ uniquely the inner normal derivative $n_i$.
			
			We assume $g$ to be such that the $H^i_0$ intersect each other orthogonally. This means that the corresponding $n_i$ are orthogonal to each other on the intersections of the $H^i_0$, which are just the boundary manifolds. Thus, we can choose boundary functions $\rho=(\rho_1,\dotsc,\rho_n)$ which satisfy $n_i(\rho_j)=\delta_{ij}$ on $H_0^i$.
			
			Let $H_0$ be any boundary manifold and $o\in H_0$. Choose a coordinate system $u_0=(\rho_{H_0},\tilde u_0)$ on a neighbourhood $U_0$ of $o$, which satisfies
			\begin{align*}
				g_{o'}\big(\der/{u_{k,0}}|_{o'},\der/{u_{j,0}}|_{o'}\big)=\delta_{kj}\quad\text{for all }o'\in H_0\cap U_0.
			\end{align*}
			Then $g_{kj}^{u_0}\big|_{H_0\cap U_0}=1$, and hence $\omega_g|_{H_0,\rho}=\big|d\tilde u_0|_{H_0\cap U_0}\big|$.
			
			On the other hand, $g$ induces a metric $g_{H_0}$ on $H_0$, which gives the canonical density
			$\omega_{g_{H_0}}=\big| d\tilde u_0|_{H_0\cap U_0}\big|$ on $H_0\cap U_0$. Thus, 
			\begin{align*}
				\omega_g|_{H_0,\rho}=\omega_{g_{H_0}}.
			\end{align*}
	\end{enumerate}
\end{remarks}

\begin{lemma}\label{la:5.16}
	Let $\gamma$ be a retraction. The restriction to a boundary supermanifold is compatible with $\gamma_!$, in the following sense: for $H\in B\bg(\gamma^*(\rho))$, 
	\begin{align*}
		{\gamma_H}_!\lr(\omega|_{H,\gamma^*(\rho)})=(-1)^{s(\dim N)+s(\dim H)}(\gamma_!(\omega))|_{H_0,\rho}
	\end{align*}
\end{lemma}

\begin{proof}
	We complete $\rho_{H_0}$ locally to a coordinate system $u_0=(\rho_{H_0},\tilde u_0)$ and choose $x=(u,\xi)=(\gamma^*(\rho_{H_0}),\tilde x)$ to be a coordinate system with which $\gamma$ is associated. We write $\omega=f|Dx|$ with $f=\sum_\nu\gamma^*(f_\nu)\xi^\nu$. Then
	\begin{align*}
		{\gamma_H}_!\lr(\omega|_{H,\gamma^*(\rho)})
			&={\gamma_H}_!\Big(\sum_\nu\iota_H^*\circ\gamma^*(f_\nu)\iota_H^*(\xi)^\nu|D\iota^*(\tilde x)|\Big)\\
			&={\gamma_H}_!\Big(\sum_\nu\gamma_H^*\circ\iota_{H_0}^*(f_\nu)\iota_H^*(\xi)^\nu|D\iota^*(\tilde x)|\Big)\\
			&=(-1)^{s(\dim H)}\,\iota_{H_0}^*\bg(f_\lr(1,\dotsc,1))|d\iota_{H_0}^*(\tilde u_0)|\\
			&=(-1)^{s(\dim H)}\big(f_\lr(1,\dotsc,1)|du_0|\big)\big|_{H_0,\rho}\\
			&=(-1)^{s(\dim N)+s(\dim H)}(\gamma_!(\omega))|_{H_0,\rho}.\qedhere
	\end{align*}
\end{proof}

\begin{lemma}\label{la:5.17}
	For the boundary superfunctions $\tau$ there exists a family of derivations $D=(D_1,\dotsc,D_n)$ with the following properties:
	\begin{itemize}
		\item
			Each $D_i$ is defined on a neighbourhood $U^D_i$ of $H_0$, where $H_0$ is the boundary manifold of codimension $1$, which is given by $\tau_{i,0}$.
		\item
			$D_i(\tau_j)=\delta_{ij}$ on $U^D_i\cap U^D_j$.
	\end{itemize}
\end{lemma}
We call a family $D$ which satisfies these conditions a \emph{family of boundary superderivations} for $\tau$. Observe that $D$ is \emph{not} uniquely determined. 

\begin{proof}
	Let $o\in M_0$ and choose a coordinate neighbourhood $U$ with $o\in U_0$. Let $\tau_{U}$ be the subfamily of $\tau$ of all boundary functions which vanish at any point of $U_0$:
		\begin{align*}
			\tau_i\in\tau_U\iff\exists o'\in U_0:\tau_{i,0}(o')=0.
		\end{align*}
		Possibly after shrinking $U$, we may assume that $\tau_U$ can be completed to a coordinate system on $U$.
		
		Since $o$ was arbitrary, we can choose a locally finite covering $\lr(U^\alpha)_{\alpha\in A}$ of $M$ with such coordinate neighbourhoods. On each $U^\alpha$ we supplement $\tau_{U^\alpha}$ to a coordinate system $x^\alpha=(\tau_{U^\alpha},\tilde x^\alpha)$ and define for $i=1,\dotsc,n$
		\begin{align*}
			D_i^\alpha\defi
				\begin{cases}
					\der/{x_s^\alpha}	& x_s^\alpha=\tau_i\in\tau_{U^\alpha},\\
					0					& \text{otherwise.}
				\end{cases}
		\end{align*}
		
		This means $D_i^\alpha(\tau_j)=\delta_{ij}$ for $\tau_j\in\tau_{U^\alpha}$. Now we choose a partition of unity $\lr(\phi_\alpha)_{\alpha\in A}$ subordinate to $\lr(U^\alpha)_{\alpha\in A}$ and glue these local derivations to
		\begin{align*}
			D_i\defi\sum_{\alpha\in A}\phi_\alpha D_i^\alpha,\quad i=1,\dotsc,n.
		\end{align*}
		
		It remains to define the neighbourhoods $U_i^D$. We set for each $i$
		\begin{align*}
			B_i\defi\left\{\alpha\in A\,\big|\,0\in\tau_{i,0}(U_0^\alpha)\right\},\quad
			C_i\defi\left\{\alpha\in A\,\big|\,0\notin\tau_{i,0}(U_0^\alpha)\right\}
		\end{align*}
		to define
		\begin{align*}
			U_i^D
				\defi\bigcup_{\beta\in B_i}U_0^\beta\bigg\backslash\bigcup_{\alpha\in C_i}\supp\phi_\alpha
				=\bigcup_{\beta\in B_i}\bigcap_{\alpha\in C_i}U_0^\beta\Big\backslash\supp\phi_\alpha.
		\end{align*}
		
		The so defined sets are open, since the covering was locally finite. By construction $\supp\phi_\alpha\cap U_i^D\cap U_j^D=\emptyset$ for all $\alpha\in C_i\cup C_j$, hence
		\begin{align*}
			1|_{U_i^D\cap U_j^D}
				=\sum_{\alpha\in A}\phi_\alpha|_{U_i^D\cap U_j^D}
				=\sum_{\alpha\in B_i\cap B_j}\phi_\alpha|_{U_i^D\cap U_j^D}.
		\end{align*}
		
		Note that $\alpha\in B_i\cap B_j$ implies $\tau_i,\tau_j\in\tau_{U^\alpha}$, hence $D_i^\alpha(\tau_j)=\delta_{ij}$. With this observation we finish the proof by calculating on $U_i^D\cap U_j^D$
		\begin{align*}
			D_i(\tau_j)|_{U_i^D\cap U_j^D}
				&=\sum_{\alpha\in A}\phi_\alpha|_{U_i^D\cap U_j^D}D_i^\alpha(\tau_j)
				=\sum_{\alpha\in B_i\cap B_j}\phi_\alpha|_{U_i^D\cap U_j^D}D_i^\alpha(\tau_j)\\
				&=\sum_{\alpha\in B_i\cap B_j}\phi_\alpha|_{U_i^D\cap U_j^D}\delta_{ij}
				=\delta_{ij}.\qedhere
		\end{align*}
\end{proof}

For $j\in\multi(n,0)$ we define $D^j\defi D_1^{j_1}\dotsm D_n^{j_n}$ and the reduced multi-index $j\mathord\downarrow\defi(j_1\mathord\downarrow,\dotsc,j_n\mathord\downarrow)$, where $s\mathord\downarrow\defi\max(s-1,0)$ for $s\in\NN_0$.

\begin{theorem}[Change of variables formula]\label{thm:trafo}\pdfbookmark[2]{Changes of variables for Berezin densities 1}{trafo1}
	Let $\omega\in\Ber M$ be a compactly supported Berezin density and let $\gamma,\gamma'$ be retractions on $M$. Then
	\begin{align*}
		\int_\lr(N,\gamma')\omega=\int_\lr(N,\gamma)\omega+\!\sum_{H\in B(\gamma^*(\rho))}\sum_{j\in J_H}(-1)^{s(\dim N)+s(\dim H)}\int_\lr(H,\gamma_H)\big(\omega_j.D^{j\downarrow}\big)\big|_{H,\gamma^*(\rho)},
	\end{align*}
	where $s(\dotsm)$ was defined in \autoref{def:3.3},  $D=(D_1,\dotsc,D_n)$ is a family of boundary superderivations for $\gamma^*(\rho)$ as in \autoref{la:5.17}, and
	\begin{align*}
		\omega_j\defi\fr 1/{j!}\left(\gamma'^*(\rho)-\gamma^*(\rho)\right)^j\omega.
	\end{align*}
\end{theorem}

Note $\omega_j=0$ if $j>\fr q/2$, so the sum over $J_H$ is finite. Observe further that there are no summands for $\codim H_0>\fr q/2$, if $q<2p$. Moreover, the Berezin density $\omega_j.D^{j\downarrow}$ is defined on $U_{i_1}^D\cap\dotsc\cap U_{i_k}^D$, where $\rho_{H_0}=(\rho_{i_1},\dotsc,\rho_{i_k})$. This set contains $H_0$, so the restriction makes sense. 

In general, there is no canonical choice for the boundary superderivations; given a super Riemannian metric on $M$, one might take boundary superderivations which are orthogonal to the boundary with respect to this metric (\cf \autoref{rk:5.16}\! \eqref{rk:5.16.2}\!).

\begin{proof}
	We will prove the formula in several steps.
	
	\emph{Step~1.}
		We suppose $M$ to be a superdomain, $M\subset\RR^{p,q}$, and $N$ to satisfy $N_0=(\RR_+^k\!\times\RR^{p-k})\cap M_0$. The boundary functions are chosen to be $\rho=(\pr_1,\dotsc,\pr_k)$. Furthermore, we consider a coordinate system $x=(u,\xi)$ with $u_0=(\pr_1,\dotsc,\pr_p)$ and $\gamma^*(u_0)=u$.
		\renewcommand{\qedsymbol}{} 
\end{proof}

\begin{lemma}\label{la:5.19}\mbox{}
 	\begin{enumerate}[(i)]
 		\item
 			For $s>k$ we have
 				$\int_\lr(N,\gamma)\omega.\der/{x_s}=0$.
 		\item
 			For $H\in B\bg(\gamma^*(\rho))$ and $s>k$ we have
 				$\int_\lr(H,\gamma_H)\big(\omega.\der/{x_s}\big)\big|_{H,\gamma^*(\rho)}=0$.
 		\item
 			If $i\in\multi(p,0)$ such that $i_s\neq 0$ for some $s>k$, then
 				$\int_\lr(N,\gamma)\omega.\diff u^j=0$.
 		\item
 			If $H\in B\bg(\gamma^*(\rho))$ and $j\in J_H$ we get (with $(j,0)\in\multi(p,0)$)
 			\begin{align*}
 				\int_\lr(N,\gamma)\omega.\diff u^{(j,0)}
 					=(-1)^{s(\dim N)+s(\dim H)}\int_\lr(H,\gamma_H)\big(\omega.\diff u^{j\downarrow}\big)\big|_{H,\gamma^*(\rho)}.
 			\end{align*}
 	\end{enumerate}
\end{lemma}

\begin{proof}[Proof of \autoref{la:5.19}]
	Write $\omega=f|Dx|$ and $f=\sum_\nu\gamma^*(f_\nu)\xi^\nu$. Since $\omega$ is compactly supported, the same is true for the $f_\nu$.
	
	(i)	If $k<s\leq p$, we deduce
		\begin{align*}
			\int_\lr(N,\gamma)\omega.\der/{u_s}=\pm\int_0^\infty do_1\dotsm\int_0^\infty do_k\int_{-\infty}^\infty do_{k+1}\dotsm\int_{-\infty}^\infty do_p\ \derf f_\lr(1,\dotsc,1)(o)/s.
		\end{align*}
		Thanks to the compact support, the right hand side vanishes. In the case of $s>p$ the claim is clear.
		
	(ii)
		Similar to (i).
		
	(iii)
		Follows directly from (i).
		
	(iv) 
		Write $\tilde u_0\defi(\pr_{k+1},\dotsc,\pr_p)$. As in \autoref{ex:5.6}\nobreak\eqref{ex:5.6.3} we apply the Fundamental Theorem of Calculus in each direction in which a derivation occurs. So the remaining integrals are the same as integrating along $H_0$. 
		
		In the following computation, we write $\ell\defi\codim H_0$.
		
		\begin{align*}
			\int_\lr(N,\gamma)\omega.\partial_u^j
				&=(-1)^{s(\dim N)+|j|}\int_0^\infty do_1\dotsm\int_0^\infty do_k\int_{-\infty}^\infty do_{k+1}\dotsm\int_{-\infty}^\infty do_p\,\partial_{u_0}^jh(o)\\
				&=(-1)^{s(\dim N)+\ell+|j|}\int_{H_0}(\partial_{u_0}^{j\downarrow} h)|_{H_0}\big|d\tilde u_0|_{H_0}\big|\\
				&=(-1)^{s(\dim N)+\ell+|j|}\int_{H_0}\big(\partial_{u_0}^{j\downarrow} h\,|du_0|\big)\big|_{H_0,\rho}\\
				&=\int_{H_0}\big(\gamma_!(\omega.\partial_u^{j\downarrow})\big)\big|_{H_0,\rho}
				=(-1)^{s(\dim N)+s(\dim H)}\int_{H_0}{\gamma_H}_!\big((\omega.\partial_u^{j\downarrow})|_{H,\gamma^*(\rho)}\big)\\
				&=(-1)^{s(\dim N)+s(\dim H)}\int_\lr(H,\gamma_H)\big(\omega.\partial_u^{j\downarrow}\big)\big|_{H,\gamma^*(\rho)}.
		\end{align*}
	
		In the second last equation, we applied \autoref{la:5.16}.
\end{proof}

\begin{proof}[Proof of \autoref{thm:trafo} (continued)]
		Remembering $\multi(k,0)\backslash\{0\}=\dot\bigcup_H J_H$, we see that \autoref{la:5.19}, applied to \autoref{thm:5.4}, proves \autoref{thm:trafo} in the case of $D_s=\der /{u_s}$ for $s=1,\dotsc,k$.
		
	\emph{Step~2.}
		We stay in the setting of Step~1, but now we suppose general $D=(D_1,\dotsc,D_k)$ such that $D_i\bg(\gamma^*(\rho_l))=\delta_{il}$ everywhere on $M$ for $i,l=1,\dotsc,k$. Then we have for $i=1,\dotsc,k$
		\begin{align*}
			D_i=\sum_{l=1}^{p+q}D_i(x_l)\der/{x_l}=\der/{u_i}+\sum_{l=k+1}^{p+q}D_i(x_l)\der/{x_l}.
		\end{align*}
		
		Therefore we get for $H\in B\bg(\gamma^*(\rho))$ and $j\in J_H$
		\begin{align*}
			D^j=\diff u^{(j,0)}+\sum_{l=k+1}^{p+q}A_l\circ\der/{x_l}
		\end{align*}
		for some (not further specified) differential operators $A_{k+1},\dotsc,A_{p+1}$. Applying \autoref{la:5.19}, we conclude
		\begin{align*}
			\int_\lr(H,\gamma_H)\omega.\lr(A_l\circ\der/{x_l})=\int_\lr(H,\gamma_H)(\omega.A_l).\der/{x_l}=0,
		\end{align*}
		hence
		\begin{align*}
			\int_\lr(H,\gamma_H)\omega.D^j=\int_\lr(H,\gamma_H)\omega.\diff u^{(j,0)}.
		\end{align*}
		
		This proves the claim for general $D=(D_1,\dotsc,D_k)$.

	\emph{Step~3.}
		Now suppose $M$ to be a coordinate neighbourhood with boundary functions $\rho=(\rho_1,\dotsc,\rho_k)$ such that $\gamma^*(\rho)$ can be completed to a coordinate system $x=\bg(\gamma^*(\rho),\tilde x)$. Suppose further $D=(D_1,\dotsc,D_k)$ to be such that $D_i\bg(\gamma^*(\rho_l))=\delta_{il}$ everywhere on $M$. Let $\varphi\colon\Omega'\to M$ be the inverse of a chart such that $\rho_i\circ\varphi_0=\pr_i,\ i=1,\dotsc,k$. We set $\Omega\defi\Omega'|_{\varphi_0^{-1}(N_0)}$. The restriction $\varphi\colon\Omega\to N$ is also an isomorphism. Furthermore, $\Omega_0=\Omega'_0\cap\RR_+^k\!\times\RR^{p-k}$, meaning that $\Omega'$ and $\Omega$ are as in Steps 1 and 2. We recognise that the boundary manifolds are sent to each other by $\varphi_0$, more explicitly
		\begin{align*}
			B_0(\pr_1,\dotsc,\pr_k)=\big\{\varphi_0^{-1}(H_0)\,\big|\,H_0\in B_0(\rho)\big\}.
		\end{align*}
		Indeed, $\varphi_0$ restricts to diffeomorphisms of the boundary manifolds.

		Similarly, $\varphi$ induces isomorphisms of the boundary supermanifolds. For each $H\in B\bg(\gamma^*(\rho))$ we denote by $H_\Omega\in B\big(\varphi^*(\gamma^*(\rho))\big)$ the supermanifold over $\varphi_0^{-1}(H_0)$, which corresponds to the retraction $\varphi^*(\gamma)$.

		Now fix $H\in B\bg(\gamma^*(\rho))$. We recall that $\iota_H^*(\tilde x)$ is a coordinate system on $H$. Moreover, $\iota_{H_\Omega}^*\bg(\varphi^*(\tilde x))$ is a coordinate system on $H_\Omega$, since $\iota_{H_\Omega}^*\big(\varphi^*(\gamma^*(\rho))\big)=0$. Therefore
		\begin{align*}
			\varphi_H^*\bg(\iota_H^*(\tilde x))\defi\iota_{H_\Omega}^*\bg(\varphi^*(\tilde x))
		\end{align*}
		defines an isomorphism $\varphi_H\colon H_\Omega\to H$ making the following diagram commute:
		\begin{align*}
			\begin{xy}
				\xymatrix{
					H\ar[r]^{\iota_H}									& M\\
					H_\Omega\ar[r]_{\iota_{H_\Omega}}\ar[u]^{\varphi_H}	& \Omega'\ar[u]_\varphi
				}
				\end{xy}
		\end{align*}
		
		The definition of $\varphi_H$ is also compatible with the pullback of retractions and with the restriction of Berezin densities. The former means just $\varphi_H^*(\gamma_H)=\bg(\varphi^*(\gamma))_{H_\Omega}$, which can be seen from the following diagram:
		\begin{align*}
			\begin{xy}
				\xymatrix@=7ex@!0{
																													& H \ar[rr]^{\iota_H}\ar[dd]|{\phantom{*}}^/-3ex/{\gamma_H}	&																&M \ar[dd]^\gamma	\\
					H_\Omega \ar[rr]_/3ex/{\iota_{H_\Omega}}\ar[dd]_{\varphi_H^*(\gamma_H)}\ar[ur]^{\varphi_H}	&																& \Omega'\ar[dd]^/-3ex/{\varphi^*(\gamma)}\ar[ur]^/-1ex/\varphi						\\
																													& H_0 \ar@{^{(}->}[rr]|{\phantom{*}}							&																&M_0				\\
					\varphi_0^{-1}(H_0) \ar@{^{(}->}[rr]\ar[ur]_/1ex/{\varphi_{H_0}}									&																& \Omega'_0 \ar[ur]_{\varphi_0}
			}\end{xy}
		\end{align*}
		
		The left, right, upper, lower, and rear squares commute, and hence this has also to be true for the front side. The uniqueness condition in \autoref{la:5.12} implies $\varphi_H^*(\gamma_H)=\bg(\varphi^*(\gamma))_{H_\Omega}$.

		With $\omega=fDx$, the compatibility of $\varphi_H$ with restriction of Berezin densities is easily derived from
		\begin{align*}
			\varphi_H^*\big(\omega\big|_{H,\gamma^*(\rho)}\big)
				&=\varphi_H^*\big(\iota_H^*(f)D\iota_H^*(\tilde x)\big)
				=\iota_{H_\Omega}^*\bg(\varphi^*(f))\,D\iota_{H_\Omega}^*\bg(\varphi^*(\tilde x))\\
				&=\big(\varphi^*(\omega)\big)\big|_{H_\Omega,\varphi^*(\gamma^*(\rho))}.
		\end{align*}

		These properties lead to
		\begin{align*}
			\int_\lr(H,\gamma_H)\omega|_{H,\gamma^*(\rho)}
				=\int_{\big(H_\Omega,(\varphi^*(\gamma))_{H_\Omega}\big)}\big(\varphi^*(\omega)\big)\big|_{H_\Omega,\varphi^*(\gamma^*(\rho))}.
		\end{align*}

		Using the compatibility of pullbacks with the action of differential operators on Berezin densities (\Cref{eq:13}), we see for $j\in J_H$
		\begin{align*}
			\int_\lr(H,\gamma_H)\big(\omega_j.D^{j\downarrow}\big)\big|_{H,\gamma^*(\rho)}
				&=\int_{\big(H_\Omega,(\varphi^*(\gamma))_{H_\Omega}\big)}\big(\varphi^*\big(\omega.D^{j\downarrow}\big)\big)\big|_{H_\Omega,\varphi^*(\gamma^*(\rho))}\\
				&=\int_{\big(H_\Omega,(\varphi^*(\gamma))_{H_\Omega}\big)}\big(\varphi^*(\omega).\varphi^*(D)^{j\downarrow}\big)\big|_{H_\Omega,\varphi^*(\gamma^*(\rho))}.
		\end{align*}
		
		We notice that $\varphi^*(D)=\big(\varphi^*(D_1),\dotsc,\varphi^*(D_k)\big)$ with $\varphi^*(D_i)=\varphi^*\circ D_i\circ\left.\varphi^*\right.^{-1}$ is a family of boundary derivations for $\varphi^*\bg(\gamma^*(\rho))=\bg(\varphi^*(\gamma))^*(\pr)$, where $\pr=(\pr_1,\dotsc,\pr_k)$. Furthermore,
		\begin{align*}
			\varphi^*(\omega_j)=\fr 1/{j!}\Big(\bg(\varphi^*(\gamma'))^*(\pr)-\bg(\varphi^*(\gamma))^*(\pr)\Big)^{j}\varphi^*(\omega).
		\end{align*}
		
		This means that we are able to use Step~2 to obtain
		\begin{align*}
			\int_\lr(N,\gamma')\omega
				&=\int_{(\Omega,\varphi^*(\gamma'))}\varphi^*(\omega)
				=\int_\lr(N,\gamma)\omega
					\,+\sum_H\,\sum_{\hidewidth j\in J_H\hidewidth}\,\pm\int_\lr(H,\gamma_H)\big(\omega_j.D^{j\downarrow}\big)\big|_{H,\gamma^*(\rho)}.
		\end{align*}

	\emph{Step~4.} 
		We prove the general formula. Let $o\in M_0$ and choose a coordinate neighbourhood $U$ with $o\in U_0$. Denote by $\rho_{U_0}$ the subfamily of $\rho$ of all boundary functions which vanish at any point of $U_0$. Let $D_U\defi\lr(D_i)_{\rho_i\in\rho_{U_0}}$ denote the corresponding family of boundary derivations.
		
		Possibly shrinking $U$ we may assume that $\gamma^*(\rho_{U_0})$ can be completed to a coordinate system on $U$. Shrinking $U$ further we can suppose $U$, together with $\rho_{U_0}$ and $D_U$, to be as in Step~3.

		Since $o$ was arbitrary, we can take a covering $\lr(U^\alpha)_\alpha$ of $M$ with such coordinate neighbourhoods. Let $\lr(\phi_\alpha)_\alpha$ be a partition of unity subordinate to $\lr(U^\alpha)_\alpha$. Then $\phi_\alpha\omega$ is compactly supported in $U^\alpha$, hence by Step~3
		\begin{align*}
			\int_\lr(N,\gamma')\omega
				&=\sum_\alpha\int_\lr(V^\alpha,\gamma')\phi_\alpha\omega\\
				&=\sum_\alpha\int_\lr(V^\alpha,\gamma)\phi_\alpha\omega\,+\sum_\alpha\,\sum_H\,\sum_{\hidewidth j\in J_H\hidewidth}\,\pm\int_\lr(H,\gamma_H)\big((\phi_\alpha\omega_j).D_{U^\alpha}^{j\downarrow}\big)\big|_{H,\gamma^*(\rho_{U^\alpha_0})},
		\end{align*}
		where $V^\alpha\defi U^\alpha\cap N$ and $H$ runs through $B\bg(U^\alpha,\gamma^*(\rho_{U^\alpha_0}))$.
		
		We recognise that
		\begin{alignat*}{2}
			B(U^\alpha_0,\rho_{U^\alpha_0})&=\Big\{U^\alpha_0\cap H_0\,&&\Big|\,H_0\in B(M_0,\rho),\ U^\alpha_0\cap H_0\neq\emptyset\Big\}
		\intertext{and similarly,}
			B\bg(U^\alpha,\gamma^*(\rho_{U^\alpha_0}))&=\Big\{H|_{U^\alpha_0\cap H_0}\,&&\Big|\,H\in B\bg(M,\gamma^*(\rho)),\ U^\alpha_0\cap H_0\neq\emptyset\Big\}.
		\end{alignat*}
		
		Further, one trivially has for $H_0\in B(M_0,\rho)$ with $U^\alpha_0\cap H_0\neq\emptyset$ that
		\begin{align*}
			J_{U^\alpha_0\cap H_0}^{\rho_{U^\alpha_0}}=\left\{(j_i)_{\rho_i\in\rho_{U^\alpha_0}}\middle|\,j\in J_{H_0}^\rho\right\}.
		\end{align*}
		
		Therefore, we get
		\begin{alignat*}{2}
			\int_\lr(N,\gamma')\omega
				&=\int_\lr(N,\gamma)\omega\,+&\,\sum_\alpha\,&\sum_H\,\sum_{\hidewidth j\in J_H\hidewidth}\,\pm\int_{(H|_{U^\alpha_0\cap H_0},\gamma_H)}\big((\phi_\alpha\omega_j).D^{j\downarrow}\big)\big|_{H,\gamma^*(\rho)},
		\intertext{where $H$ runs now through $B\bg(M,\gamma^*(\rho))$. The supports of the integrands are contained in $U^\alpha_0\cap H_0$, which implies}
			\int_\lr(N,\gamma')\omega
				&=\int_\lr(N,\gamma)\omega\,+&\,\sum_\alpha\,&\sum_H\,\sum_{\hidewidth j\in J_H\hidewidth}\,\pm\int_\lr(H,\gamma_H)\big((\phi_\alpha\omega_j).D^{j\downarrow}\big)\big|_{H,\gamma^*(\rho)}\\
				&=\int_\lr(N,\gamma)\omega\,+&&\sum_H\,\sum_{\hidewidth j\in J_H\hidewidth}\,\pm\int_\lr(H,\gamma_H)\big(\omega_j.D^{j\downarrow}\big)\big|_{H,\gamma^*(\rho)}.
		\end{alignat*}
	
	This completes the proof of \autoref{thm:trafo}.
\end{proof}

Sometimes, the following statement of the change of variables formula, which contains only ordinary boundary derivations, is more useful in applications.

\begin{corollary}\label{cor:5.21}\pdfbookmark[2]{Changes of variables for Berezin densities 2}{trafo1.5}
	Let $A=(A_1,\dotsc,A_n)$ be a family of boundary derivations on $M_0$ corresponding to $\rho$, \ie $A_i(\rho_l)=\delta_{il}$ on appropriate neighbourhoods. Then
	\begin{align*}
		\int_\lr(N,\gamma')\omega=\int_\lr(N,\gamma)\omega+\sum_{H_0\in B_0(\rho)}\sum_{j\in J_{H_0}}\int_{H_0}\big((\gamma_!\omega_j).A^{j\downarrow}\big)\big|_{H_0,\rho},
	\end{align*}
	where $\omega_j$ is at in \autoref{thm:trafo}.
\end{corollary}

\begin{proof}
	One only needs to check for $H\in B\bg(\gamma^*(\rho))$ and $j\in J_H$
	\begin{align*}
		\big((\gamma_!\omega).A^j\big)\big|_{H_0,\rho}=(-1)^{s(\dim N)+s(\dim H)}{\gamma_H}_!\big((\omega.D^j)|_{H,\gamma^*(\rho)}\big),
	\end{align*}
	where $D=(D_1,\ldots,D_n)$ is a the family of boundary superderivations as in \autoref{thm:trafo}.
	
	This can be done locally. Write $\omega=fDx$ with $x=(u,\xi)=\bg(\gamma^*(\rho),\tilde x)$. Then, similarly to Step~2 of the proof of \autoref{thm:trafo}, one sees
	\begin{align*}
		D_i=\der/{u_s}+\sum_{l=k+1}^{p+q} a_l\der/{x_l},\quad A_i=\der/{u_{s,0}}+\sum_{l=k+1}^p b_l\der/{u_{l,0}}
	\end{align*}
	for some $s$ and $k=\codim H_0$, hence, using an analogous argument as in Step~3 of the proof of \autoref{thm:trafo},
	\begin{align*}
		(\gamma_!\omega).A_i\big|_{H_0,\rho}
			&=(\gamma_!\omega).\der/{u_{s,0}}\big|_{H_0,\rho}
			=\gamma_!\lr(\omega.\der/{u_s})\big|_{H_0,\rho}
			=\gamma_!\lr(\omega.D_i)\big|_{H_0,\rho}.
	\end{align*}
	
	It follows from \autoref{la:5.16} that
	\begin{align*}
		(\gamma_!\omega).A^j\big|_{H_0,\rho}
			&=\bg(\gamma_!(\omega.D^j))\big|_{H_0,\rho}
			=\pm{\gamma_H}_!\big((\omega.D^j)|_{H,\gamma^*(\rho)}\big).\qedhere
	\end{align*}
\end{proof}

\begin{examples}\label{ex:5.22}\mbox{}
	\begin{enumerate}[(i)]
		\item
			Let $N\subset M=\RR^{2,4}$ be the superdomain with $N_0=\ ]0,1[^2$ and $\gamma,\gamma'$ be retractions on $M$. As we have seen earlier, boundary functions for $N_0$ are given by $\rho=(\pr_1,\pr_2,1-\pr_1,1-\pr_2)$. Adequate boundary derivations might be $D=(\der/{u_1},\der/{u_2},-\der/{u_1},-\der/{u_2})$, where we choose $x=(u,\xi)=\big(\gamma^*(\pr_1),\gamma^*(\pr_2),\xi\big)$.
			
			We count the number of summands which can be non-zero. Since $q=4$ we obtain $1$ summand for each dimension $0$ boundary manifold, and $2$ summands for each of dimension $1$, resulting in $1+4\cdot2+4\cdot1=13$ summands.
		\item\label{ex:5.22.2}
			We reconsider  \autoref{ex:5.6}\! \eqref{ex:5.6.2}\! in light of the above theorem. Recall the notation introduced in the example. Let $\gamma$ be the retraction given by $x$ and $\gamma'$ be the retraction associated with $y$. We compute $\gamma^*(v_0)$ on $\Omega$:
			\begin{align*}
				\gamma^*(v_{1,0})
					&=\gamma^*\bg(u_{1,0}\cos(u_{2,0}))
						=u_1\cos(u_2)=\fr v_1/{1-\xi_1\xi_2}\\
						&=v_1\lr(1+\fr\eta_1\eta_2/{v_1^2+v_2^2}),\\
				\gamma^*(v_{2,0})
					&=v_2\lr(1+\fr\eta_1\eta_2/{v_1^2+v_2^2}).
			\end{align*}
			
			This shows that $\gamma$ can be continued to a retraction on $\Omega''$, where $\Omega''_0=\Omega_0\backslash\{0\}$. Unfortunately, $\Omega''$ cannot be considered as manifold with corners in the ambient space $\Omega_0$. 
			
			For this reason, let $\Omega_\varepsilon\subset\Omega$ for $0<\varepsilon<1$ be given by $\Omega_{\varepsilon,0}=\{(o_1,o_2)\mid \varepsilon^2<o_1^2+o_2^2<1\}$. We turn $\Omega_\varepsilon$ to a supermanifold with corners in $\Omega$ \via the boundary function $\rho$ given by
			\begin{align*}
				\rho=r-\varepsilon,\quad r=\sqrt{v_{1,0}^2+v_{2,0}^2}\,.
			\end{align*}
			Of course, $r=u_{1,0}$ on $\Omega'$.
			
			Since $q=2$, we do not need to find any boundary derivation (although it is easy to see that the radial operator $v_{1,0}\der/{v_{1,0}}+v_{2,0}\der/{v_{2,0}}$ is a boundary derivation). Let $f\in\OO(\Omega)$ be compactly supported. Using
			\begin{align*}
				\gamma'^*(\rho)-\gamma^*(\rho)
					&=\sqrt{v_1^2+v_2^2}-\sqrt{v_1^2+v_2^2}\lr(1+\fr\eta_1\eta_2/{v_1^2+v_2^2})\\
					&=-\fr\eta_1\eta_2/{\sqrt{v_1^2+v_2^2}}
					=-\fr\eta_1\eta_2/{\gamma^*(r)}
			\end{align*}
			we get by \autoref{cor:5.21}
			\begin{alignat*}{2}
				\int_\lr(\Omega,y)f|Dy|
					&=\lim_{\varepsilon\to0}\int_\lr(\Omega_\varepsilon,\gamma')&&f|Dy|\\
					&=\lim_{\varepsilon\to0}\int_\lr(\Omega_\varepsilon,\gamma)&&f|Dy|+\int_{\varepsilon S^1}\gamma_!\left(-\fr\eta_1\eta_2/{\gamma^*(r)}f|Dy|\right)\bigg|_{\varepsilon S^1,\rho}.
			\end{alignat*}
			
			For the application of $\gamma_!$ we need to use $|D\hat y|$ with $\hat y=\bg(\gamma^*(v_0),\eta)$. A simple calculation shows $\Da \hat y/y=1$. Furthermore, we recognise that $|dv_0|$ is the standard density on $\RR^2$. In \autoref{rk:5.16}\! \eqref{rk:5.16.2}\! we saw that this just means $|dv_0||_{\varepsilon S^1}=dS$, which leads to
			\begin{align*}
				\gamma_!\left(-\fr\eta_1\eta_2/{\gamma^*(r)}f|Dy|\right)\bigg|_{\varepsilon S^1,\rho}
					=-(-1)^{s(2,2)}\fr f_0/r dv_0\Big|_{\varepsilon S^1,\rho}
					=-(-1)^{s(2,2)}\fr f_0/\varepsilon dS.
			\end{align*}
			
			Since
			\begin{align*}
				\lim_{\varepsilon\to 0}\int_{\varepsilon S^1}\fr f_0/\varepsilon dS=2\pi f_0(0),
			\end{align*}
			we see again
			\begin{align*}
				\int_\lr(\Omega,y)f|Dy|
					=\int_\lr(\Omega'',\gamma)f|Dy|-(-1)^{s(2,2)}2\pi f_0(0).
			\end{align*}
	\end{enumerate}
\end{examples}

One may give a version of the change of variables formula for Berezin forms by considering induced orientations on the boundary manifolds. To do so, one has to fix an ordering of the boundary functions $\rho$ and has to keep track of the boundary orientations; on the boundary manifolds of dimension $0$, this leads to additional signs. 

We do not state the resulting formula in full generality, since it is somewhat cumbersome. However, in the case of a supermanifold with boundary (\ie for the case of only one boundary supermanifold), the theorem can be easily restated for Berezin forms, as follows. 

\begin{corollary}\pdfbookmark[2]{Changes of variables for Berezin forms}{trafo2}\label{cor:5.19}
	Let $U\subset M$ with smooth boundary $\partial U_0$ and let $\gamma,\gamma'$ be retractions on $M$. Then for compactly supported $\omega\in\Ber M$ we have
	\begin{align*}
		\int_\lr(U,\gamma')\omega
			=\int_\lr(U,\gamma)\omega
				-\pm\sum_{j=1}^{\lfloor\fr q/2\rfloor}\fr 1/{j!}\int_\lr(\partial_\gamma U,\partial\gamma)\big(\big(\big(\gamma'^*(\rho)-\gamma^*(\rho)\big)^j\omega\big).D^{j-1}\big)\big|_{\partial_\gamma U,\gamma^*(\rho)}.
	\end{align*}
	Here, $\rho$ is a boundary function for $U_0$ and $D$ is a boundary derivation corresponding to $\gamma^*(\rho)$. The sign $\pm$ is given by $(-1)^{s(p,q)+s(p-1,q)}$.
\end{corollary}

The additional minus sign occurring in the above formula comes from the fact that the boundary derivations define \emph{inner} normals. 

In \autoref{ex:4.6}, we saw that it is important to choose the right immersion $\iota\colon\partial_\gamma U\to M$ to arrive at the usual formulation of Stokes's theorem. In the case of an arbitrary boundary supermanifold $\iota\colon\partial U\to M$, one can apply the corollary to show the following generalisation of Stokes's theorem (where, of course, additional boundary terms have to appear). 

\begin{corollary}\label{cor:stokes2}\pdfbookmark[2]{Stokes formula for a fixed boundary structure}{stokes2}
	Let $U\subset M$ be a supermanifold with boundary such that $\overline U_0$ be compact, and let $\omega=d\varpi$ be an exact Berezin form. Then
	\begin{align*}
		\int_\lr(U,\gamma')\omega
			=(-1)^{s(p,q)+s(p-1,q)}\bigg(
				(-1)^q\int_{\partial U}\iota^*(\varpi)
				-\sum_{j=1}^{\lfloor\fr q/2\rfloor}\int_{\partial U}\left(\omega_j.D^{j-1}\right)\big|_{\partial U,\tau}
				\bigg),
	\end{align*}
	where $\tau$ is chosen such that $\iota^*(\tau)=0$ and $\omega_j\defi\fr1/{j!}\bg(\gamma'^*(\tau_0)-\tau)^j\omega$.
\end{corollary}
Note that in this formula, the retraction $\gamma$ does not occur any longer.

\providecommand{\bysame}{\leavevmode\hbox to3em{\hrulefill}\thinspace}
\providecommand{\MR}{\relax\ifhmode\unskip\space\fi MR }
\providecommand{\MRhref}[2]{%
  \href{http://www.ams.org/mathscinet-getitem?mr=#1}{#2}
}
\providecommand{\href}[2]{#2}

\end{document}